\documentclass[11pt,leqno,a4paper]{amsart}

\usepackage{hyperref}
\usepackage{latexsym}
\usepackage{amssymb,amsmath}
\usepackage{graphics}
\usepackage{url}
\usepackage[vcentermath]{youngtab}
  
\usepackage{booktabs}  
\usepackage{caption}
\usepackage{tikz}
\tikzstyle{empty}=[circle,draw=black!80,thick]
\tikzstyle{emptyn}=[circle,draw=black!80,fill=white,scale=0.5] 
\tikzstyle{nero}=[circle,draw=black!80,fill=black!80,thick]

\usepackage{amssymb,enumerate}
\newcommand{\fS}{{\mathfrak{S}}}

\newcommand{\Irr}{{\operatorname{Irr}}}

\newtheorem{thm}{Theorem}[section]
\newtheorem{lem}[thm]{Lemma}

\newtheorem{cor}[thm]{Corollary}
\newtheorem{prop}[thm]{Proposition}

\newtheorem*{thmA}{Theorem A}
\newtheorem*{thmB}{Theorem B}
\newtheorem*{thmC}{Theorem C}
\newtheorem*{thmD}{Theorem D}

\newtheorem*{thm1}{Theorem 1}

\theoremstyle{definition}

\newtheorem{defn}[thm]{Definition}

\newtheorem{notation}[thm]{Notation}
\newtheorem{rem}[thm]{Remark}

\begin{document}

\title{On the $p'$-subgraph of the Young graph}


\author{Eugenio Giannelli}
\address[E. Giannelli]{Department of Pure Mathematics and Mathematical Statistics, University of Cambridge, Cambridge CB3 0WA, UK}
\email{eg513@cam.ac.uk}

\author{Stacey Law}
\address[S. Law]{Department of Pure Mathematics and Mathematical Statistics, University of Cambridge, Cambridge CB3 0WA, UK}
\email{swcl2@cam.ac.uk}

\author{Stuart Martin}
\address[S. Martin]{Department of Pure Mathematics and Mathematical Statistics, University of Cambridge, Cambridge CB3 0WA, UK}
\email{sm137@cam.ac.uk}

\thanks{The first author's research was funded by Trinity Hall, University of Cambridge.}
 
\begin{abstract}
Let $p$ be a prime number. In this article we study the restriction to $\fS_{n-1}$ of irreducible characters of degree coprime to $p$ of $\fS_n$. In particular, we study the combinatorial properties of the subgraph $\mathbb{Y}_{p'}$ of the Young graph $\mathbb{Y}$. This is an extension to odd primes of the work done in \cite{APS} for $p=2$.
\end{abstract}



\maketitle

\section{Introduction}\label{sec:intro}

Let $\mathcal{P}$ denote the set of partitions of natural numbers. For $\lambda$ a partition of $n$ and $\mu\in\mathcal{P}$ we let $(\lambda, \mu)\in\mathcal{E}$ if and only if $\chi^\mu$ is an irreducible constituent of $(\chi^\lambda)_{\fS_{n-1}}$. Here we denoted by $\chi^\lambda$ the ordinary irreducible character of the symmetric group $\fS_n$ naturally labelled by $\lambda$ (this notation will be kept throughout the article). 
The Young graph $\mathbb{Y}$ has $\mathcal{P}$ as its set of vertices and $\mathcal{E}$ as its set of edges. 

In the representation theory of symmetric groups, the study of the Young graph has proved a fruitful tool in the modern development of the subject.
For example, the recent approach to this area first presented in \cite{OV} (see also \cite[Chapter 2]{KBook}) derives the entire representation theory of symmetric groups from the combinatorial properties of $\mathbb{Y}$.
It is somewhat surprising that only very recently in \cite{APS}, the following remarkable fact was shown to hold. 
\begin{thm1}[Unique Parent Theorem in \cite{APS}]\label{sn} 
Let $n\in\mathbb{N}$ and let $\chi$ be an irreducible character of odd degree of $\fS_n$. Then the restriction $\chi_{\fS_{n-1}}$ has a unique irreducible constituent of odd degree.
\end{thm1}

Let $p$ be a prime number and let $\mathbb{Y}_{p'}$ be the induced subgraph of $\mathbb{Y}$ on those vertices (partitions) labelling irreducible characters of degree coprime to $p$. Theorem~\ref{sn} shows that $\mathbb{Y}_{2'}$ is a rooted tree. 
Starting from this beautiful observation, the rest of \cite{APS} is devoted to describing the \textit{highly} regular combinatorial structure of $\mathbb{Y}_{2'}$. 
We remark that the relevance of \cite{APS} transcends the study of the Young graph. In fact, Theorem~\ref{sn} was recently used to construct several types of character correspondences (see \cite{BGO}, \cite{GKNT} and \cite{INOT}).

\medskip

The main aim of this paper is to study the combinatorial structure of $\mathbb{Y}_{p'}$ for any odd prime $p$. As remarked in \cite[Section 7]{APS}, Theorem 1 is false for odd primes and $\mathbb{Y}_{p'}$ is never a tree for $p\geq 3$.
Yet notably, given \emph{any} prime $p$ and any irreducible character $\chi$ of degree coprime to $p$ of $\fS_n$, Theorem A (below) describes the number of irreducible constituents of degree coprime to $p$ of $\chi_{\fS_{n-1}}$. In particular, this extends Theorem 1 to all primes.

For $\lambda\in\mathcal{P}(n)$ we denote by $\lambda^-_{p'}$ the subset of $\mathcal{P}(n-1)$ consisting of all partitions $\mu$ such that $\chi^\mu$ is an irreducible constituent of degree coprime to $p$ of $(\chi^\lambda)_{\fS_{n-1}}$. Moreover we define $\mathcal{E}_n$ to be the set $$\mathcal{E}_n=\left\{|\lambda^-_{p'}|\ :\ \lambda\vdash n\ \ \text{and}\ \ p\nmid\chi^\lambda(1)\right\},$$
and we let $br(n)$ be the maximal value in $\mathcal{E}_n$ (i.e. $br(n)=\mathrm{max}\{|\lambda^-_{p'}|\ :\ \lambda\vdash n\ \ \text{and}\ \ p\nmid\chi^\lambda(1)\}$).
Our first result describes $\mathcal{E}_n$ and gives a recursive formula for the exact value of $br(n)$.

\begin{thmA}
Let $n\in\mathbb{N}$ and let $p$ be a prime.  Let $n=\sum_{j=1}^t a_jp^{n_j}$ be the $p$-adic expansion of $n$, for some $0\leq n_1<n_2<\cdots <n_t$.
Then $\mathcal{E}_n=\{1,2,\ldots, br(n)-1, br(n)\}$ and $$br(n)=br(a_1p^{n_1})+\sum_{j=2}^t\Phi(a_j, br(m_j))$$
where $m_j=\sum_{i=1}^{j-1} a_ip^{n_i}$, and where $\Phi$ is the function described explicitly in Definition~\ref{def:PhiPsi} below.
\end{thmA}

In Section~\ref{sec:apk} we determine exactly $br(ap^k)$ for any prime $p$, any $k\in\mathbb{N}_0$ and any $a\in\{1,\ldots, p-1\}$. The following result serves as the base case for computing $br(n)$ for any natural number $n$, using the recursive expression given in Theorem A.

\begin{thmB}
Let $p$ be an odd prime, $k\in\mathbb{N}_0$ and $a\in\{1,\dotsc,p-1\}$. 
Then $$br(ap^k)= \begin{cases}
f(2a) & \mathrm{if}\ k=0,\\
p-1 + 2\lfloor \frac{2a-(p-1)}{6}\rfloor & \mathrm{if}\ k=1\ \mathrm{and}\ \frac{p}{2}<a<p,\\
2a & \mathrm{otherwise}.
\end{cases}$$
Here $f(x)=\mathrm{max}\{y\in\mathbb{N}_0\ |\ y(y+1)\leq x\}.$
\end{thmB}

Theorem B is stated only for odd primes since we know that if $p=2$, then $br(2^k)=1$ for all $k\in\mathbb{N}_0$, by \cite{APS}.

Theorems A and B provide us with a recursive formula for $br(n)$, the maximal number of downward edges from level $n$ to level $n-1$ of $\mathbb{Y}_{p'}$. 
In the second part of our article we show that the slightly involved expression for the value of $br(n)$ described in Theorem A can be bounded from above by a much nicer function of the $p$-adic digits of $n$. 

\begin{thmC}
Let $n\in\mathbb{N}$ and let $p$ be a prime.  Let $n=\sum_{j=1}^t a_jp^{n_j}$ be the $p$-adic expansion of $n$, for some $0\leq n_1<n_2<\cdots <n_t$.
Then $1\leq br(n)\leq \mathcal{B}_n$, where 
$$\mathcal{B}_n = br(a_1p^{n_1}) + \sum_{j=2}^t \left\lfloor \frac{a_j}{2} \right\rfloor\leq 2a_1+ \sum_{j=2}^t \left\lfloor \frac{a_j}{2} \right\rfloor.$$ 
\end{thmC}

Theorem C has some interesting direct applications (see Section~\ref{sec:D}). For instance,
in Remark~\ref{rem:appl} below, we observe that when $p\in\{2,3\}$ then $\mathcal{B}_n=br(n)$. In particular our result is a generalization of Theorem~\ref{sn}. Moreover, for any prime $p$ we observe that the upper bound $\mathcal{B}_n$ is attained for every $n$ having all of its $p$-adic digits lying in $\{0,1,2,3\}$.

We further show that the upper bound $\mathcal{B}_n$ given in Theorem C is indeed a good approximation of $br(n)$. 
In fact, the following result shows that the difference $\varepsilon_n:=\mathcal{B}_n-br(n)$ can be bounded by a constant depending only on the prime $p$, and not on $n\in\mathbb{N}$.

\begin{thmD}
For any $n\in\mathbb{N}$, we have $\varepsilon_n < \frac{p}{2}\mathrm{log}_2(p)$.
\end{thmD}

A consequence of Theorem D is that for any odd prime $p$ we have $\mathrm{sup}\{br(n)\ :\ n\in\mathbb{N}\}=\infty$. This is false when $p=2$, by Theorem 1.


\section{Notation and Preliminaries}\label{sec:2}
In this section we fix the notation that will be used throughout the article and recall some basic facts in the representation theory of symmetric groups (we refer the reader to \cite{JK} and \cite{OlssonBook} for detailed accounts of the theory). We begin by introducing the technical notation necessary to state and prove Theorem A.
\begin{defn}\label{def:PhiPsi}
For $a\in\mathbb{N}_0$ and $L\in\mathbb{N}$, define
$$\Phi(a,L):=\max\left\{\sum_{i=1}^Lf(a_i)\ \middle|\ a_1+\cdots+a_L\le a\ \ \mathrm{and}\ \ a_i\in\mathbb{N}_0\ \ \forall\ 1\le i\le L 
\right\},$$ where $f(x)=\mathrm{max}\{y\in\mathbb{N}_0\ |\ y(y+1)\leq x\}.$
\end{defn}

We now record some properties of this function $\Phi$ which will be useful for later proofs.

\begin{lem}\label{lem:PhiBound}
Let $a\in\mathbb{N}_0$ and $L\in\mathbb{N}$. Then $\Phi(a,L)\le\lfloor a/2\rfloor$. In particular, if $L\ge\lfloor a/2\rfloor$ then $\Phi(a,L)=\lfloor a/2\rfloor$.
\end{lem}
\begin{proof}
Observe that for all integers $x\ge 2$, we have $f(x)\le f(2)+f(x-2)$. Hence $$\Phi(a,L)\le \lfloor a/2\rfloor\cdot f(2) + f(\delta)$$ where $\delta\in\{0,1\}$ and $\delta\equiv a\ (\operatorname{mod} 2)$. But $f(2)=1$ and $f(1)=f(0)=0$, so the assertions follow.
\end{proof}

\begin{lem}\label{lem:2powers}
Let $k\in\mathbb{N}$. Then $2^{k-1}\le \Phi(2^k+2,2^{k-1})\le 2^{k-1}+1$.
\end{lem}
\begin{proof}
When $k=1$, we note that $\Phi(4,1)=1$. Now assume $k\ge 2$. The upper bound follows from Lemma~\ref{lem:PhiBound}. The lower bound follows from the fact that $2^k+2 = 6 + 2\cdot(2^{k-1}-2)+0$, and $f(6)+f(2)\cdot(2^{k-1}-2)+f(0) = 2^{k-1}$.
\end{proof}

\subsection{Combinatorics of partitions}\label{subs:comb}
Let $n$ be a natural number. We denote by $\mathcal{P}(n)$ the set of partitions of $n$ and we let $$\mathcal{P}=\bigcup_{n\in\mathbb{N}}\mathcal{P}(n).$$  The notation $\lambda\in\mathcal{P}(n)$ is sometimes replaced by $\lambda\vdash n$. For any natural number $e$, we denote by $C_e(\lambda)$ and $Q_e(\lambda)=(\lambda_0,\lambda_1,\ldots, \lambda_{e-1})$ the $e$-core and the $e$-quotient of $\lambda$ respectively. The $e$-weight of $\lambda$ is the natural number $w_e(\lambda)$ defined by 
$w_e(\lambda)=|\lambda_0|+|\lambda_1|+\cdots+|\lambda_{e-1}|.$
We remark that given a partition $\lambda$ of $n$, the $e$-quotient $\mathcal{Q}_e(\lambda)$ is uniquely determined up to a cyclic shift of its components. 
Moreover, it is well-known that (up to the above mentioned shift) any partition is uniquely determined by its $e$-core and $e$-quotient (we refer the reader to \cite{OlssonBook} for a detailed discussion on the topic).

Let $\mathcal{H}_e(\lambda)$ be the set of hooks of $\lambda$ having length divisible by $e$ and let $\mathcal{H}(Q_e(\lambda))=\cup_{i=1}^e\mathcal{H}(\lambda_{i})$. As explained in \cite[Theorem 3.3]{OlssonBook}, there is a bijection between $\mathcal{H}_e(\lambda)$ and $\mathcal{H}(Q_e(\lambda))$ mapping hooks in $\lambda$ of length $ex$ to hooks in the quotient of length~$x$.
Moreover the bijection respects the process of hook removal. Namely, the partition $\mu$ obtained by removing a $ex$-hook from $\lambda$ is such that $C_e(\mu)=C_e(\lambda)$ and the $e$-quotient of $\mu$ is obtained by removing a $x$-hook from one of the $e$ partitions involved in the $e$-quotient 
of~$\lambda$.
The other fundamental result we need to recall is \cite[Proposition 3.6]{OlssonBook}, which can be stated as follows. 

\begin{prop}\label{prop:hooks}
Let $\lambda\in\mathcal{P}(n)$. The number of $e$-hooks that must be removed from $\lambda$ to obtain $C_e(\lambda)$ is $w_e(\lambda)$. Moreover 
ì$w_e(\lambda)=|\mathcal{H}_e(\lambda)|=(|\lambda|-|C_e(\lambda)|)/e.$
\end{prop}

\noindent\textbf{James' Abacus.} All of the operations on partitions concerning addition and removal of $e$-hooks described above are best performed on James' abacus. We give here a brief description of this important object, and introduce some pieces of notation that will be used extensively throughout. We refer the reader to \cite[Chapter 2]{JK} for a complete account of the combinatorial properties of James' abacus.

Let $\lambda$ be a partition of $n$ and let $A$ be an $e$-abacus configuration for $\lambda$. Denote by $A_0,$ $A_1,$ $\ldots,$ $A_{e-1}$ the runners in $A$ from left to right and label the rows by integers such that the row numbers increase downwards. As is customary, all abaci contain finitely many rows and hence finitely many beads, but in all instances enough to perform all of the necessary operations. For $j\in\{0,\ldots, e-1\}$, denote by $|A_j|$ the number of beads on runner $j$. Moreover, we denote by $A^{\uparrow}$ the $e$-abacus obtained from $A$ by sliding all beads on each runner as high as possible. Extending the notation just introduced, we denote by $A_0^{\uparrow},\ldots, A_{e-1}^{\uparrow}$ the runners of $A^{\uparrow}$.   
As explained in \cite[Chapter 2]{JK}, $A^{\uparrow}$ is an $e$-abacus for $C_{e}(\lambda)$. Let the operation of sliding any single bead down (resp.~up) one row on its runner be called a \emph{down-move} (resp.~\emph{up-move}). Of course, such a move is only possible for a bead in position $(i,j)$ (that is, in row $i$ on runner $A_j$) if the respective position $(i\pm 1,j)$ was empty initially. Sometimes we call an empty position a \textit{gap}. We say that position $(x,y)$ is the \textit{first} gap in $A$ if there are beads in positions $(i,j)$ for all $i<x$ and all $j$, and in positions $(x,j)$ for all $j<y$.

On the level of partitions, performing a down- or up-move corresponds to adding or removing an $e$-hook, respectively. In analogy with the notation used for partitions, we denote by $w(A)$ the total number of up-moves needed to obtain $A^\uparrow$ from $A$. Similarly, for $i\in\{0,\ldots, e-1\}$ we let $w(A_i)$ be the number of those up-moves that were performed on runner $i$ in the transition from $A$ to $A^\uparrow$. It is easy to see that $w_e(\lambda)=w(A)=w(A_0)+\cdots +w(A_{e-1})$. 

Suppose that $c$ is a bead in position $(i,j)$ of $A$. We say that $c$ is a \textit{removable bead} if $j\ne 0$ and there is no bead in $(i,j-1)$, or if $j=0$ and there is no bead in $(i-1,e-1)$. Denote by $A^{\leftarrow c}$ the abacus configuration obtained by sliding $c$ into position $(i,j-1)$ (respectively $(i-1,e-1)$).
Clearly $A^{\leftarrow c}$ is an abacus configuration for a partition $\mu\in\lambda^-$, and conversely any $\mu\in\lambda^-$ can be represented by $A^{\leftarrow c}$ for some such $c$, since removable beads in an abacus of $\lambda$ correspond to removable nodes in the Young diagram of $\lambda$. Here and throughout the remainder of the article we denote by $\lambda^-$ the subset of $\mathcal{P}(n-1)$ consisting of all partitions whose Young diagram can be obtained from that of $\lambda$ by removing a node.

Finally, for $j\in\{0,\ldots, e-1\}$ we denote by $\mathrm{Rem}(A_j)$ the number of removable beads in $A$ lying on runner $A_j$. In particular, we have that $|\lambda^-|=\mathrm{Rem}(A_0)+\cdots+\mathrm{Rem}(A_{e-1})$.

\begin{lem}\label{lem:weightgeneral}
Let $e\in\mathbb{N}$. Let $\lambda$ be a partition of any natural number, and denote by $A$ an $e$-abacus configuration for $\lambda$. Suppose $c$ is a removable bead on runner $A_j$ and let $\mu\vdash n-1$ be the partition represented by $A^{\leftarrow c}$. Then
\[ w_e(\mu) = w_e(\lambda) + \begin{cases}
|A_j| - |A_{j-1}| - 1 & \mathrm{if}\ j\ne 0\\
|A_0| - |A_{e-1}| - 2 & \mathrm{if}\ j = 0.\\
\end{cases} \]
\end{lem}
\begin{proof}
First suppose $j\ne 0$. Without loss of generality we can relabel the rows of the $e$-abacus $A$ such that all rows labelled by negative integers do not have empty positions. 
To ease the notation we let $B:=A^{\leftarrow c}$.
Clearly $w(A_i)=w(B_i)$ for all $i\ne j-1, j$ in $\{0,\dotsc,e-1\}$. Hence $$ w_e(\mu)-w_e(\lambda) = w(B_{j-1})+w(B_j) - w(A_{j-1})-w(A_j).$$
Let $s$ and $t$ be the numbers of beads lying in rows labelled by non-negative integers in runners $A_{j-1}$ and $A_j$ respectively. Suppose that the $s$ beads on $A_{j-1}$ lie in rows $0\le x_1<\cdots<x_s$ and that the $t$ beads on $A_j$ lie in rows $0\le y_1<\cdots<y_t$. Then $$w(A_{j-1})+w(A_j) = \sum_{i=1}^s(x_i-(i-1)) + \sum_{i=1}^t (y_i-(i-1)) = \sum_{i=1}^s x_i + \sum_{i=1}^t y_i -\tfrac{s(s-1)}{2} - \tfrac{t(t-1)}{2}.$$
Suppose that the bead $c$ lies in row $y_l$ for some $l\in\{1,\dotsc,t\}$. Since $c$ is removable, $y_l\ne x_i$ for all $i\in\{1,\dotsc,s\}$. Thus the beads on $B_{j-1}$ lie in rows $0\le x'_1<\cdots<x'_{s+1}$ with $\{x'_1,\dotsc,x'_{s+1}\}=\{x_1,\dotsc,x_s,y_l \}$ and the beads on $B_j$ lie in rows $0\le y'_1<\cdots<y'_{t-1}$ with $\{y'_1,\dotsc,y'_{t-1} \}=\{y_1,\dotsc,y_{l-1},y_{l+1},\dotsc,y_t \}$. Hence $$w(B_{j-1}) + w(B_j) = \sum_{i=1}^{s+1}(x'_i-(i-1)) + \sum_{i=1}^{t-1}(y'_1-(i-1)) = \sum_{i=1}^s x_i + \sum_{i=1}^t y_i - \tfrac{s(s+1)}{2} - \tfrac{(t-1)(t-2)}{2}$$ and we conclude that $w_e(\mu)-w_e(\lambda)=t-s-1=|A_j|-|A_{j-1}|-1$. 

The case when $j=0$ is similar.
\end{proof}

\begin{rem}
In this note, given a partition $\lambda$ and a fixed $e$-abacus configuration $A$ for $\lambda$ we let $\lambda_i$ be the partition corresponding to the runner $A_i$, considered as a $1$-abacus. 
The resulting $e$-quotient $(\lambda_0,\lambda_1,\ldots, \lambda_{e-1})$ depends on the choice of the abacus $A$ (a different choice of the $e$-abacus may induce a cyclic shift on the components of the $e$-quotient). Nevertheless, all of the results presented in Section \ref{sec:2} onwards hold independently of this observation. For instance, the $e$-weight $w_e(\lambda)$ introduced at the beginning of Section \ref{subs:comb} does not depend on the choice of the $e$-abacus; the same discussion holds for Theorem \ref{thm:oddcrit2} below.
\end{rem}

\subsection{Characters of $\fS_n$}
For each $n\in\mathbb{N}$, the elements of the set $\mathrm{Irr}(\fS_n)$ of irreducible characters of $\fS_n$ are naturally labelled by partitions of $n$.
For $\lambda\in\mathcal{P}(n)$, the corresponding irreducible character is denoted by $\chi^\lambda$. In this article we will often identify the labelling partition with the corresponding irreducible character, and hence write $\lambda\in\Irr(\fS_n)$ to denote at once the partition $\lambda$ of $n$ and the irreducible character $\chi^\lambda$. The meaning of this notation will always be clear from the context. We recall the \textit{Branching rule} (see \cite[Chapter 9]{James}) which tells us that $$(\chi^\lambda)_{\fS_{n-1}}=\sum_{\mu\in\lambda^-}\chi^\mu.$$
By convention we let $\mathfrak{S}_0$ be the trivial 1-element group and $\mathcal{P}(0)=\{\emptyset\}$.

\smallskip

From now on let $p$ be a prime. We denote by $\Irr_{p'}(\fS_n)$ the set of irreducible characters of $\fS_n$ of degree coprime to $p$. We say that $\lambda$ is a $p'$-partition of $n$ (written $\lambda\vdash_{p'} n$) if $\lambda\in\Irr_{p'}(\fS_n)$. Thus the set $\mathcal{E}_n$ may be written as $\{|\lambda^-_{p'}|\ :\ \lambda\vdash_{p'} n \}$, and we remark here that $br(n)=\max\mathcal{E}_n$ is well-defined. Indeed, for any $n\in\mathbb{N}$, if $\lambda\vdash_{p'}n$ then $|\lambda^-_{p'}|\ge 1$, so $\mathcal{E}_n$ is non-empty and $br(n)\ge 1$.

\smallskip

Irreducible characters of $\fS_n$ of $p'$-degree were completely described in \cite{Mac}. We restate this result in language that will be particularly convenient for our purposes.
\begin{thm} \label{thm:oddcriterion}
Let $n$ be a natural number and let $\lambda\in\mathrm{Irr}(\fS_n)$. Let $a\in\{1,\ldots, p-1\}$ and $k\in\mathbb{N}_0$ be such that $ap^k\leq n<(a+1)p^k$. Then $\lambda\in\mathrm{Irr}_{p'}(\fS_n)$ if and only if $C_{p^k}(\lambda)\in\mathrm{Irr}_{p'}(\fS_{n-ap^k})$. 
\end{thm}

Theorem~\ref{thm:oddcriterion} says that $\lambda$ is a $p'$-partition if and only if $w_{p^k}(\lambda)=a$ and the partition $C_{p^k}(\lambda)$ obtained from $\lambda$ by successively removing all possible $p^k$-hooks is a $p'$-partition of $n-ap^k$. It will sometimes be useful to use the following equivalent version of Theorem~\ref{thm:oddcriterion}.

\begin{thm}\label{thm:oddcrit2}
Let $n=\sum_{j=0}^ka_jp^j$ be the $p$-adic expansion of $n\in\mathbb{N}$. 
Let $\lambda\in\mathrm{Irr}(\fS_n)$ and let $Q_{p}(\lambda)=(\lambda_0,\lambda_1,\ldots, \lambda_{p-1})$. Then $\lambda\in\mathrm{Irr}_{p'}(\fS_n)$ if and only if
\begin{itemize}
\item[(i)] $C_p(\lambda)\vdash a_0$, and 
\item[(ii)] for all $t\in\{0,1,\ldots, p-1\}$ there exists $b_{1t}, b_{2t},\ldots, b_{kt}\in\mathbb{N}_0$ such that 
$$\sum_{t=0}^{p-1}b_{jt}=a_j\ \text{for all}\ j\in\{1,\ldots k\},\text{and such that}\ \ \lambda_t\vdash_{p'}\sum_{j=1}^kb_{jt}p^{j-1}.$$  
\end{itemize}
\end{thm}
\begin{proof}
This characterization of $p'$-partitions of $n\in\mathbb{N}$ can be easily proved using the $p$\textit{-core tower} associated to any partition of $n$. We refer the reader to \cite[Chapters I and II]{OlssonBook} for the precise description of this combinatorial object. 
\end{proof}


\section{The core map and the proofs of Theorems A and C}\label{sec:AC}

In this section we prove some combinatorial statements that will play a fundamental role in the proofs of all of our main theorems. As a consequence of these observations, we are able to give proofs of Theorems A and C. As appropriately remarked later in this section, the proof of Theorem B is postponed to Section~\ref{sec:apk} to improve readability. 
\begin{notation}\label{not:sec3}
Unless otherwise stated, in this section we fix $n\in\mathbb{N}$ such that $n=ap^k+m$ for some $k\ge 1$, $a\in\{1,\ldots,p-1\}$ and $0<m<p^k$. To be precise this will be the standing assumption from Lemma \ref{lem:weight} to Proposition \ref{prop:br_m}.
\end{notation}

\begin{lem}\label{lem:weight}
Let $\lambda\vdash n$ be such that $w_{p^k}(\lambda)=w\leq a$ and denote by $A$ a $p^k$-abacus configuration for $\lambda$. Suppose $c$ is a removable bead on runner $A_j$ and let $\mu\vdash n-1$ be the partition represented by $A^{\leftarrow c}$. 
Then $w_{p^k}(\mu)=w$ if and only if 
\[ |A_j|\ =\begin{cases}
 1+|A_{j-1}| & \mathrm{if}\ j\neq 0,\\ 2+|A_{p^k-1}| & \mathrm{if}\ j=0.\end{cases} \]
\end{lem}
\begin{proof}
This is immediate by Lemma~\ref{lem:weightgeneral}.
\end{proof}

The following result, which we believe is of independent interest, is one of the key steps in proving Theorem A.

\begin{thm}\label{thm:coremap}
Let $\lambda\vdash_{p'}n$ and let $\alpha\in\lambda_{p'}^{-}$. Then $C_{p^k}(\alpha)\in \mu_{p'}^{-}$, where $\mu:=C_{p^k}(\lambda)$. 
In particular we deduce that the map $$C_{p^k}: \lambda_{p'}^-\longrightarrow\mu_{p'}^-,$$ is well-defined. Moreover, it is surjective.
\end{thm}
\begin{proof}
Let $A$ be the $p^k$-abacus configuration for $\mu$ having first gap in position $(0,0)$. It is easy to see that rows $i\ge 1$ must be empty, since $|\mu|=m<p^k$. (We will not need rows $i$ with $|i|>a$, so we may assume row $-a$ is the top row of the abacus and $+a$ the bottom row.) So $|A_0|=a$ and $a\leq |A_j|\leq a+1$, for all $j\in\{0,\ldots, p^k-1\}$. Let $B$ be the $p^k$-abacus configuration for $\lambda$ such that $B^\uparrow=A$. By Proposition~\ref{prop:hooks}, we have $w_{p^k}(\lambda)=a$ and we see that $B$ is obtained from $A$ after performing exactly $a$ down-moves.

Let $c$ be the bead in $B$ such that $B^{\leftarrow c}$ is an abacus configuration for $\alpha$, and suppose $c$ lies on runner $B_j$. Since $\alpha$ is a $p'$-partition of $n-1=ap^k+(m-1)\geq ap^k$ we deduce from Theorem~\ref{thm:oddcriterion} that $w_{p^k}(\alpha)=a$. Hence by Lemma~\ref{lem:weight} we have $|B_j|=1+|B_{j-1}|$ ($j$ cannot be $0$ because $|B_l|=|A_l|\in\{a,a+1\}$ for all $l\in\{0,\ldots, p^k-1\}$). It follows that there exists a bead $d$ in position $(0, j)$ of $A$ and that position $(0,j-1)$ of $A$ is empty. Hence $A^{\leftarrow d}$ is a $p^k$-abacus configuration for $C_{p^k}(\alpha)$, which by Theorem~\ref{thm:oddcriterion} must be a $p'$-partition. Thus $C_{p^k}(\alpha)\in\mu^-_{p'}$ and the map $C_{p^k}: \lambda_{p'}^-\longrightarrow\mu_{p'}^-$ is well-defined.

To show that the map is surjective we proceed as follows. Let $A$ be the $p^k$-abacus configuration for $\mu$ as described above. For any $\beta\in\mu^-_{p'}$ there exists a bead $d$ in $A$ such that $A^{\leftarrow d}$ is a $p^k$-abacus configuration for $\beta$. Let $j\in\{1,\ldots, p^k-1\}$ be such that $d$ is in position $(0,j)$ in $A$ and such that position $(0,j-1)$ is empty.
Let $B$ be the $p^k$-abacus for $\lambda$ described above. Clearly we have that $|B_j|=|A_j|=1+|A_{j-1}|=1+|B_{j-1}|$. Hence there exists a row $y\in\{-a,\dotsc,a\}$ such that position $(y,j-1)$ of $B$ is empty and such that there is a bead (say $e$) in position $(y, j)$. Let $\alpha$ be the partition corresponding to the $p^k$-abacus $B^{\leftarrow e}$. By Lemma~\ref{lem:weight} we deduce that $w_{p^k}(\alpha)=a$. Moreover it is clear that $C_{p^k}(\alpha)=\beta\in\mathrm{Irr}_{p'}(\fS_{n-ap^k})$. By Theorem~\ref{thm:oddcriterion} we deduce that $\alpha\in\lambda^-_{p'}$ and therefore $C_{p^k}$ is surjective.
\end{proof}

\begin{cor}\label{cor:lowerbound}
Let $\lambda\vdash_{p'}n$. Then $|C_{p^k}(\lambda)_{p'}^-|\leq |\lambda_{p'}^-|$.
\end{cor}

Keeping $n=ap^k+m$ as in Notation~\ref{not:sec3}, we now introduce the following notation. Given $\gamma\vdash_{p'}m$, define $$br(n,\gamma):=\mathrm{max}\{|\lambda^-_{p'}|\ :\ \lambda\vdash_{p'}n\ \ \text{and}\ \ C_{p^k}(\lambda)=\gamma\}.$$
Clearly $br(n)$, the main object of our study, is equal to the maximal $br(n,\gamma)$, where $\gamma$ is any $p'$-partition of $m$.
Corollary~\ref{cor:lowerbound} allows us to give the following definition.

\begin{defn}\label{def:N}
Let $n=ap^k+m$ be as in Notation~\ref{not:sec3}, and let $\gamma\vdash_{p'}m$. We define $N(a, p^k, \gamma)\in\mathbb{N}_0$ to be such that $|\gamma_{p'}^-|+N(a, p^k, \gamma)=br(n,\gamma).$
\end{defn}

One of the main goals of the present section is to prove the following fact. 

\begin{prop}\label{prop:Nvalue}
Let $\gamma\vdash_{p'}m$ and let $L=|\gamma_{p'}^-|$. Then 
$N(a,p^k,\gamma) = \Phi(a, L),$
where $\Phi$ is as described in Definition~\ref{def:PhiPsi}.
\end{prop}

In order to prove Proposition~\ref{prop:Nvalue}, we need to introduce the following combinatorial concepts.  

\begin{defn}\label{def:Red}
Let $n=ap^k+m$ be as in Notation~\ref{not:sec3}, and let $\gamma\vdash_{p'}m$. Denote by $A_{\gamma}$ the $p^k$-abacus configuration for $\gamma$ having first gap in position $(0,0)$. Define $\mathcal{R}_{A_\gamma}$ to be the subset of $\{0,1,\ldots, p^k-1\}$ such that $j\in\mathcal{R}_{A_\gamma}$ if and only if there is a removable bead $c$ on runner $j$ of $A_\gamma$ such that the partition corresponding to the $p^k$-abacus $A_\gamma^{\leftarrow c}$ is a $p'$-partition of $m-1$.
\end{defn}

Since $A_\gamma$ has first gap in position $(0,0)$ and since $|\gamma|=m<p^k$ we deduce that all removable beads in $A_\gamma$ lie in row $0$. Hence $|\mathcal{R}_{A_\gamma}|=|\gamma_{p'}^-|$. By definition of removable bead, we have in particular that $0\notin\mathcal{R}_{A_\gamma}$, and for $1\le j\le p^k-2$ we have that if $j\in\mathcal{R}_{A_\gamma}$ then $j+1\notin\mathcal{R}_{A_\gamma}$.

\begin{lem}\label{lem:red}
Let $\gamma\vdash_{p'}m$.
Let $\lambda\vdash_{p'}n$ be such that $C_{p^k}(\lambda)=\gamma$ and let $B$ be the $p^k$-abacus for $\lambda$ such that $B^{\uparrow}=A_{\gamma}$.
Let $c$ be a removable bead on runner $j$ of $B$ and let $\mu$ be the partition of $n-1$ corresponding to $B^{\leftarrow c}$. Then $\mu$ is a $p'$-partition if and only if $j\in\mathcal{R}_{A_\gamma}$.
\end{lem}
\begin{proof}
Let $A:=A_\gamma$. First suppose $j\in\mathcal{R}_{A}$. In particular, $j\ne 0$. Then
$$|B_j|=|A_j|=|A_{j-1}|+1=|B_{j-1}|+1,$$
so $w_{p^k}(\mu)=a$ by Lemma~\ref{lem:weight}. We also have that $(B^{\leftarrow c})^{\uparrow}$ is an abacus configuration for $C_{p^k}(\mu)$. Moreover if $d$ is the bead in position $(0,j)$ of $A$ then
$(B^{\leftarrow c})^{\uparrow}=A^{\leftarrow d}$. Therefore we deduce that $C_{p^k}(\mu)\in\gamma_{p'}^-$ and hence that $\mu\vdash_{p'}n-1$, by Theorem~\ref{thm:oddcriterion}.

Now suppose that $j\notin\mathcal{R}_A$.
If $j=0$ then $|B_0|=|A_0|\neq |A_{p^k-1}|+2= |B_{p^k-1}|+2$. Hence $w_{p^k}(\mu)\neq a$ by Lemma \ref{lem:weight} and therefore $\mu$ is not a $p'$-partition, by Theorem~\ref{thm:oddcriterion}.
Now we may assume that $j\neq 0$. If $|A_j|\neq |A_{j-1}|+1$ then $|B_j|\neq |B_{j-1}|+1$ and hence $w_{p^k}(\mu)\neq a$, by Lemma~\ref{lem:weight}. In particular $\mu$ is not a $p'$-partition, by Theorem~\ref{thm:oddcriterion}. If $|A_j|=|A_{j-1}|+1$, then $C_{p^k}(\mu)\in\gamma^-$ is represented by the $p^k$-abacus $(B^{\leftarrow c})^{\uparrow}$. Again we have that $(B^{\leftarrow c})^{\uparrow}=A^{\leftarrow d}$, where $d$ is the bead in position $(0,j)$ of $A$. Since $j\notin\mathcal{R}_A$ we deduce that $C_{p^k}(\mu)$ is not a $p'$-partition. It follows that $\mu\vdash n-1$ is not a $p'$-partition, by Theorem~\ref{thm:oddcriterion}.
\end{proof}

\begin{cor}\label{cor:removable beads}
Let $\gamma\vdash_{p'}m$ and let $\lambda\vdash_{p'}n$ be such that $C_{p^k}(\lambda)=\gamma$. Let $B$ be the $p^k$-abacus for $\lambda$ such that $B^{\uparrow}=A_{\gamma}$.
Then $$|\lambda^-_{p'}|=\sum_{j\in\mathcal{R}_{A_\gamma}}\mathrm{Rem}(B_j).$$
\end{cor}

Recall from Definition~\ref{def:PhiPsi} that $f(x)=\max\{y\in\mathbb{N}_0\ |\ y(y+1)\le x \}.$ The following lemma describes the key relationship between this function $f$ and certain removable beads, which will be necessary for the proof of Proposition~\ref{prop:Nvalue} (below).

\begin{lem}\label{lem:T}
Let $\lambda\in\{\emptyset,(1)\}$ and let $T_\lambda$ denote the $2$-abacus configuration of $\lambda$ having first gap in position $(0,0)$.  Let $x\in\mathbb{N}_0$ and let $\mathcal{T}_\lambda(x)$ be the set of all $2$-abaci $U$ such that $w(U)=x$ and $U^{\uparrow}=T_\lambda$. Then
$$\mathrm{max}\{\mathrm{Rem}(U_1)\ |\ U\in\mathcal{T}_\lambda(x)\}=
\begin{cases}
f(x)+1 & \mathrm{if}\ \lambda=(1),\\
\lfloor\sqrt{x}\rfloor & \mathrm{if}\ \lambda=\emptyset.
\end{cases}$$
\end{lem}

\begin{proof}
This is clear if $x=0$ or $x=1$, so we may assume now that $x\ge 2$ (and hence $f(x)>0$). We first fix $\lambda=(1)$; this is the case that we will need to use in the proof of Proposition~\ref{prop:Nvalue} below. Since $\lambda$ is now fixed, we ease the notation by letting $T_{(1)}=T$ and $\mathcal{T}_{(1)}(x)=\mathcal{T}(x)$, for all $x\in\mathbb{N}_0$. Moreover, 
let $F(x):=\mathrm{max}\{\mathrm{Rem}(U_1)\ |\ U\in\mathcal{T}(x)\}$.
We first show that there exists $A\in\mathcal{T}(x)$ such that $\mathrm{Rem}(A_1)=F(x)$ and  such that $w(A_0)=0$ (equivalently $w(A_1)=x$).

Let $U\in\mathcal{T}(x)$ be such that $w(U_0)= \ell$ and $\mathrm{Rem}(U_1)=r$ for some $l\in\{1,2,\dotsc,x\}$ and some $r\in\{0,1,\dotsc,F(x)\}$. Then there exists a $2$-abacus $V\in\mathcal{T}(y)$ for some $y\leq x$ such that $w(V_0)<\ell$ and $\mathrm{Rem}(V_1)\geq r$. 
This follows from the following observation. Since $\ell\geq 1$ there exists $i\in\mathbb{Z}$ such that there is a bead in position $(i,0)$ of $U$ and such that position $(i-1,0)$ of $U$ is empty.
Denoting beads by $\textsf{X}$ and gaps by $\textsf{O}$, consider the four possibilities for rows $i-1$ and $i$ of $U$:
$${}^{i-1}_i\qquad\quad
{}^\textsf{O}_\textsf{X} {}^\textsf{O}_\textsf{X}\qquad\quad {}^\textsf{O}_\textsf{X}{}^\textsf{O}_\textsf{O}\qquad\quad {}^\textsf{O}_\textsf{X}{}^\textsf{X}_\textsf{X}\qquad\quad {}^\textsf{O}_\textsf{X}{}^\textsf{X}_\textsf{O}$$
In the first three instances, we can move the bead in $(i,0)$ to $(i-1,0)$, to obtain the desired abacus configuration $V$. In the fourth case, we need to additionally move the bead in $(i-1,1)$ to $(i,1)$.
Hence, if $B\in\mathcal{T}(x)$ is such that $\mathrm{Rem}(B_1)=F(x)$ then there exists $y\leq x$ and $A'\in\mathcal{T}(y)$ such that $\mathrm{Rem}(A'_1)=F(x)$, $w(A'_0)=0$ and $w(A'_1)=y$. Let $(i,1)$ be the lowest position occupied by a bead (say $d$) in $A'$. Moving $d$ to position $(i+(x-y), 1)$ we obtain a $2$-abacus configuration $A\in\mathcal{T}(x)$ such that 
$\mathrm{Rem}(A_1)=\mathrm{Rem}(A'_1)=F(x)$, $w(A_0)=0$ and $w(A_1)=x$, as desired.

We want to prove that $F(x)=f(x)+1$. First suppose for a contradiction that $F(x)\ge f(x)+2$, and let $A\in\mathcal{T}(x)$ be such that $\mathrm{Rem}(A_1)=F(x)$ and $w(A_0)=0$. By construction there exists integers $0\leq j_1<j_2<\cdots<j_{f(x)+2}$ such that there is a bead in position $(j_k, 1)$ of $A$ for all $k\in\{1,\ldots, f(x)+2\}$. This implies that $w(A)=w(A_1)\geq (f(x)+1)(f(x)+2)>x$, a contradiction. Hence $F(x)\leq f(x)+1$.

Now let $y:=f(x)(f(x)+1)\leq x$. Let $B$ be the $2$-abacus configuration obtained from $T$ by first sliding down the bead in position $(0,1)$ to position $(f(x)+x-y,1)$ and then sliding down the bead in position $(i,1)$ to position $(i+f(x),1)$ for $i=-1,-2,\dotsc,-f(x)$. Clearly $B\in\mathcal{T}(x)$ and $\mathrm{Rem}(B_1)=f(x)+1$. We conclude that $F(x)=f(x)+1$, as desired. 

The case $\lambda=\emptyset$ is similar.
\end{proof}

\begin{proof}[Proof of Proposition~\ref{prop:Nvalue}]
Let $\lambda\vdash_{p'}n$ be such that $C_{p^k}(\lambda)=\gamma$ and $|\lambda_{p'}^-|=br(n,\gamma)$. Let $B$ be the $p^k$-abacus for $\lambda$ such that $B^{\uparrow}=A_{\gamma}$. In particular, $B$ is obtained from $A_\gamma$ by performing $a$ down-moves. Let $\mathcal{R}_{A_\gamma}=\{j_1,\dotsc,j_L \}$. Then by Corollary~\ref{cor:removable beads}, we have $$L + N(a,p^k,\gamma) = br(n,\gamma) = |\lambda_{p'}^-| = \sum_{i=1}^L \operatorname{Rem}(B_{j_i}).$$ 

Let $a_i=w(B_{j_i-1})+w(B_{j_i})$ for $i\in\{1,2,\dotsc,L\}$, so $a_1+\cdots+a_L\le a$. Since no two numbers in $\mathcal{R}_{A_{\gamma}}$ are consecutive (as remarked after Definition~\ref{def:Red}), we can regard the pairs of runners of $(B_{j_1-1}, B_{j_1})$, $(B_{j_2-1}, B_{j_2}),\ldots, (B_{j_L-1}, B_{j_L})$ as $L$ disjoint $2$-abaci, whose $2$-cores are all equal to the $2$-abacus $T_{(1)}$ considered in Lemma~\ref{lem:T}. It is easy to see that the $2$-abacus identified by the pair $(B_{j_i-1}, B_{j_i})$ lies in $\mathcal{T}_{(1)}(a_i)$ for all $i\in\{1,\ldots, L\}$.
Lemma~\ref{lem:T}, together with the maximality of $|\lambda_{p'}^-|$ among all the $p'$-partitions of $n$ with $p^k$-core equal to $\gamma$, allows us to deduce that 
$\mathrm{Rem}(B_{j_i})=f(a_i)+1$, for all $i\in\{1,\ldots, L\}$.
Hence 
we obtain
$$N(a,p^k,\gamma) = \sum_{i=1}^L \mathrm{Rem}(B_{j_i})-L = \sum_{i=1}^L f(a_i).$$
We conclude the proof by showing that
$$N(a,p^k,\gamma) = \max\left\{ \sum_{i=1}^L f(a'_i)\ |\ a'_1 + \cdots + a'_L \le a,\ a'_i\in\mathbb{N}_0\ \forall\ i\right\}=\Phi(a, L).$$

Suppose for a contradiction that there exists a natural number $y\leq a$ and $(a'_1,\ldots, a'_L)$ a composition of $y$ such that $\sum_{i=1}^L f(a'_i)>N(a,p^k,\gamma)$. Since $f$ is a non-decreasing function, without loss of generality we can assume that $y=a$. Then by using constructions analogous to those in the proof of Lemma~\ref{lem:T}, we can construct a partition $\tilde{\lambda}\vdash_{p'}n$ with $C_{p^k}(\tilde\lambda)=\gamma$, $w_{p^k}(\tilde\lambda)=a$ and $p^k$-abacus configuration $\tilde{B}$ satisfying $\tilde{B}^{\uparrow}=A_\gamma$ such that $w(\tilde{B}_{j_i})=a_i'$ and $\mathrm{Rem}(\tilde{B}_{j_i})=f(a'_i)+1$ for all $i\in\{1,\ldots, L\}$. This implies that $$br(n,\gamma)\geq |\tilde{\lambda}_{p'}^-|=L+\sum_{i=1}^L f(a'_i)>L+N(a,p^k,\gamma)=|\lambda_{p'}^-|=br(n,\gamma),$$
which is a contradiction. Hence $N(a,p^k,\gamma)=\Phi(a,L)$.
\end{proof}

\begin{prop}\label{prop:br_m}
Let $\gamma\vdash_{p'}m$. Then $br(n)=br(n,\gamma)$ if and only if 
$|\gamma_{p'}^-|=br(m)$. In particular, $br(n)=br(m)+\Phi(a,br(m))$.
\end{prop}

\begin{proof}
First suppose that $br(n)=br(n,\gamma)$. Let $\lambda\vdash_{p'}n$ be such that $C_{p^k}(\lambda)=\gamma$ and $|\lambda_{p'}^-|=br(n)$, so that 
$br(n) = |\gamma_{p'}^-| + \Phi(a,|\gamma_{p'}^-|)$
by Proposition~\ref{prop:Nvalue}. Let $\delta\vdash_{p'}m$ be such that $|\delta_{p'}^-|=br(m)$. Then, since $\Phi(X,Y)$ is non-decreasing in each argument (when the other argument is fixed), we have
$$br(n)\geq br(n,\delta) = |\delta_{p'}^-|+\Phi(a,|\delta_{p'}^-|)=br(m)+\Phi(a,br(m))\geq |\gamma_{p'}^-|+\Phi(a,|\gamma_{p'}^-|)=br(n),$$
whence equalities hold in the above. This proves all three statements: $br(m)=|\gamma_{p'}^-|$ gives the only if direction; $br(n)=br(n,\delta)$ gives the if direction (with $\delta$ in place of $\gamma$); and the final assertion is clear.
\end{proof}

This is enough to deduce that the second statement of Theorem A holds. 

\begin{cor}\label{cor:itera}
Let $n=\sum_{j=1}^t a_jp^{n_j}$ be the $p$-adic expansion of $n$, for some $0\leq n_1<\cdots <n_t$.
Let $m_j=\sum_{i=1}^{j-1} a_ip^{n_i}$, then $$br(n) 
= br(a_1p^{n_1}) + \sum_{j=2}^{t} \Phi(a_j, br(m_j)).$$
\end{cor}

\begin{proof}[Proof of Theorem C]
This is a straightforward consequence of Lemma~\ref{lem:PhiBound} and Corollary~\ref{cor:itera}.
\end{proof}

In the last part of this section we aim to complete the proof of Theorem A, by studying the set $\mathcal{E}_n=\{|\lambda^-_{p'}|\ :\ \lambda\vdash n\ \ \text{and}\ \ p\nmid\chi^\lambda(1)\}.$
We first state and assume the following theorem. 
\begin{thm}\label{thm:T-apk}
Let $p$ be a prime, $k\in\mathbb{N}_0$ and $a\in\{1,2,\ldots, p-1\}$. Then $\mathcal{E}_{ap^k}=\{1,2,\ldots,br(ap^k)\}$.
\end{thm}
The proof of Theorem~\ref{thm:T-apk} is rather more technical, and so has been postponed to Section~\ref{sec:apk}. More precisely, Theorem~\ref{thm:T-apk} follows from Propositions~\ref{prop:k=0},~\ref{prop:consec, 2a case} and~\ref{prop:consec, floor case}, which are proved in Section~\ref{sec:apk} below.

The next statement extends the observations already made in Lemma~\ref{lem:T}, and is crucial to completing the description of the set $\mathcal{E}_n$. 

\begin{lem}\label{lem:T2}
Let $B=T_{(1)}$ denote the $2$-abacus configuration of the partition $(1)$ having first gap in position $(0,0)$. Let $x\in\mathbb{N}_0$ and let $\mathcal{T}(x)$ be the set consisting of all $2$-abaci $U$ such that $w(U)=x$ and $U^{\uparrow}=B$. Then $\{\mathrm{Rem}(U_1)\ |\ U\in\mathcal{T}(x)\}=\{1,2,\ldots, f(x)+1\}.$
\end{lem}
\begin{proof}
From Lemma~\ref{lem:T} we know that $f(x)+1$ is the maximal value in $\{\mathrm{Rem}(U_1)\ |\ U\in\mathcal{T}(x)\}$. For any $r\in\{0,1,\ldots, f(x)\}$, let $U(r)$ be the $2$-abacus configuration obtained from $B$ by first sliding down the bead in position $(0,1)$ to position $(x-r(r+1), 1)$
and then (if $r>0$) sliding down the bead in position $(i,1)$ to position $(i+r,1)$ for $i=-1,-2,\dotsc,-r$. 
Clearly $U(r)\in\mathcal{T}(x)$ and $\mathrm{Rem}(U(r)_1)=r+1$.
\end{proof}

\begin{thm}\label{thm:consec}
Let $n\in\mathbb{N}$ and let $p$ be a prime. Let $n = \sum_{j=1}^t a_jp^{n_j}$ be the $p$-adic expansion of $n$, for some $0\le n_1<n_2<\cdots<n_t$. Then $\mathcal{E}_n = \{1,2,\dotsc,br(n)\}$.
\end{thm}
\begin{proof}
We prove the assertion by induction on $t$, the $p$-adic length of $n$. If $t=1$ then the statement follows from Theorem~\ref{thm:T-apk}.

Assume that $t\ge 2$. Let $m=\sum_{j=1}^{t-1} a_jp^{n_j}$ and let $\gamma$ be a $p'$-partition of $m$ such that $|\gamma_{p'}^-|=br(m)$. For convenience, let $L=br(m)$ and $k=n_t$. As in Definition~\ref{def:Red} let $A:=A_\gamma$ be the $p^k$-abacus configuration for $\gamma$ having first gap in position $(0,0)$. Moreover, let $\mathcal{R}_A=\{j_1,\dotsc,j_L\}$. 

Applying Lemma~\ref{lem:T2} to the $L$ pairs of runners $(A_{j_i-1},A_{j_i})$ of $A$, we see that for each $r\in\{0,1,\dotsc,\Phi(a_t,L)\}$, there exists a sequence of $a_t$ down-moves that can be performed on $A$ to produce a $p^k$-abacus $B^r$ such that $$\sum_{j\in\mathcal{R}_A} \operatorname{Rem}(B^r_j) = L+r.$$ 
Let $\lambda(r)$ be the partition of $n$ corresponding to $B^r$. Clearly $C_{p^k}(\lambda(r))=\gamma$ and by Theorem~\ref{thm:oddcriterion} we deduce that $\lambda(r)\vdash_{p'}n$. Moreover, $|\lambda(r)_{p'}^-|=L+r$ by Corollary \ref{cor:removable beads}.  Hence $L+r\in \mathcal{E}_n$, and thus $\{L,L+1,\dotsc,br(n)\}\subseteq \mathcal{E}_n$, noting that $L+\Phi(a_t,L)=br(n,\gamma)=br(n)$ by Proposition~\ref{prop:br_m}.

If $L=1$ then the proof is complete; otherwise, using the inductive hypothesis we have that for any $i\in\{1,2,\dotsc,L-1\}$, there exists $\gamma(i)\vdash_{p'}m$ such that $|\gamma(i)^-_{p'}|=i$. Taking $r=0$ and replacing $\gamma$ by $\gamma(i)$ in the above construction, we construct $\beta(i)\vdash_{p'}n$ such that $C_{p^k}(\beta(i))=\gamma(i)$ and $|\beta(i)^-_{p'}|=i+0$. Hence $\{1,2,\dotsc,L-1\}\subseteq \mathcal{E}_n$, and we conclude that $\mathcal{E}_n=\{1,2,\dotsc,br(n)\}$.
\end{proof}

\begin{proof}[Proof of Theorem A]
This follows directly from Corollary~\ref{cor:itera} and Theorem~\ref{thm:consec}.
\end{proof}

\section{The upper bound $\mathcal{B}_n$}\label{sec:D}
In this section we prove Theorem D. Let $n\in\mathbb{N}$ and let $n=\sum_{j=1}^t a_jp^{n_j}$ be the $p$-adic expansion of $n$, for some $0\le n_1<\cdots<n_t$. 
Recall that $$\mathcal{B}_n= br(a_1p^{n_1}) + \sum_{j=2}^t \left\lfloor \frac{a_j}{2} \right\rfloor.$$
From Lemma \ref{lem:PhiBound} and Corollary~\ref{cor:itera}, we see that $br(n)\leq\mathcal{B}_n$, and the difference $\varepsilon_n=\mathcal{B}_n-br(n)$ can be written as $$\varepsilon_n = \sum_{j=2}^t\left( \lfloor a_j/2\rfloor - \Phi(a_j,br(m_j)) \right)$$ where $m_j = \sum_{i=1}^{j-1} a_ip^{n_i}$. 
The following statement will be useful in the proof of Theorem D, below.
\begin{lem}\label{lem:br_incr}
Let $s,t\in\mathbb{N}_0$ with $s\le t$. Let $b_0,b_1,\dotsc,b_t\in\{0,1,\dotsc,p-1\}$ with $b_0,b_1,\dotsc,b_s$ not all zero. Then $br\left(\sum_{j=0}^s b_jp^j\right) \le br\left(\sum_{j=0}^t b_jp^j\right)$.
\end{lem}
\begin{proof}
This follows directly from Proposition~\ref{prop:br_m}. 
\end{proof}

\medskip

\begin{proof}[Proof of Theorem D]
Fix $n\in\mathbb{N}$ and its $p$-adic expansion as above. Let $\varepsilon(j)=\lfloor a_j/2\rfloor - \Phi(a_j,br(m_j))$. 
If $a_j\le 3$ then $\varepsilon(j)=0$ by Lemma~\ref{lem:PhiBound}, since $br(m_j)\ge 1$. Hence if $a_j\le 3$ for all $j\ge 2$, then in fact $\varepsilon_n=0$. In particular if $p\le 3$ then $\varepsilon_n=0$, so from now on we may assume $p\ge 5$ and that there exists $i\in\{2,\ldots, t\}$ such that $a_i\ge 4$. 
In particular, there exists a unique $k\in\{1,\ldots, t\}$ and integers $1=i_0<i_1<i_2<\cdots <i_k\leq t$ such that for all $j\in\{1,\ldots, k\}$ $$i_j:=\mathrm{min}\left\{x\in\{i_{j-1}+1,\ldots, t-1,t\}\ |\ a_x\geq 2^j+2\right\},$$
and such that $\{x\in\{i_{k}+1,\ldots, t-1,t\}\ |\ a_x\geq 2^{k+1}+2\}=\emptyset$. Note $k$ must satisfy $2^k<p$, because if $2^k\ge p$ then $a_{i_k}\ge 2^k+2>p-1$, contradicting the fact that $a_{i_k}$ is a $p$-adic digit.

We first show that $br(m_{i_j})\geq 2^{j-1}$ for all $j\in\mathbb{N}$ by induction. This is clear for $j=1$. 
For $j\in\{2,\ldots, t\}$, we have that  
$$br(m_{i_j})\geq br(m_{i_{j-1}+1})= br(m_{i_{j-1}})+\Phi(a_{i_{j-1}}, br(m_{i_{j-1}}))\geq 2^{j-2}+\Phi(2^{j-1}+2, 2^{j-2})\geq 2^{j-1}.$$
The inequalities above hold by Lemma~\ref{lem:br_incr}, the fact that $\Phi$ is non-decreasing in each argument, the inductive hypothesis and Lemma~\ref{lem:2powers}, while the equality follows from Proposition~\ref{prop:br_m}. 
Thus for all $x\geq i_{j}+1$ we have that $$br(m_{x})\geq br(m_{i_{j}+1})= br(m_{i_{j}})+\Phi(a_{i_{j}}, br(m_{i_{j}}))\geq 2^{j-1}+\Phi(2^{j}+2, 2^{j-1})\geq 2^{j}.$$

Now let $x\in\{2,\ldots, t\}$ be such that $i_j<x<i_{j+1}$ for some $j\in\{1,\ldots, k\}$. Since $x<i_{j+1}$, we have $a_x\le 2^{j+1}+1$, and since $x>i_j$, we have by the above discussion that $br(m_x)\ge 2^j$. Therefore $br(m_x)\geq \lfloor a_x/2\rfloor$ and hence $\varepsilon(x)=0$, by Lemma~\ref{lem:PhiBound}. Similarly if $x<i_1$ then $a_x\le 3$ and so $\varepsilon(x)=0$, while if $x>i_k$ then $br(m_x)\ge 2^k\ge\lfloor a_x/2\rfloor$ and thus $\varepsilon(x)=0$ also.
Hence $$\varepsilon_{n}=\sum_{j=1}^k\varepsilon(i_j).$$

Finally, for each $j\in\{1,\ldots, k\}$, we have by Lemma~\ref{lem:2powers} that
$$\varepsilon(i_j)=\lfloor a_{i_j}/2\rfloor - \Phi(a_{i_j},br(m_{i_j}))\leq \tfrac{p-1}{2}- \Phi(2^j+2, 2^{j-1})\leq \tfrac{p-1}{2}-2^{j-1}.$$
Hence
$$\varepsilon_n=\sum_{j=1}^k\varepsilon(i_j)\leq \sum_{i=0}^{k-1}(\tfrac{p-1}{2}-2^i)=
k\cdot\tfrac{p-1}{2}-(2^{k}-1)< k\cdot\tfrac{p}{2}< \tfrac{p}{2}\log_2 p.$$
\end{proof}

\begin{rem}\label{rem:appl}
Theorem D shows that the difference between the upper bound $B_n$ and the actual value of $br(n)$ is relatively small, and can be bounded independently of $n$. If $p\in\{2,3\}$ then $\varepsilon_n=0$, as observed in the first part of the proof of Theorem D above. 
In particular, fixing $p=2$ we recover \cite[Theorem 1]{APS}.
As already mentioned in the introduction, the proof of Theorem D also shows for any prime $p$, we have $\mathcal{B}_n=br(n)$ whenever all of the $p$-adic digits of $n$ are at most $3$.
\end{rem}

\section{The value of $br(ap^k)$ and the set $\mathcal{E}_{ap^k}$}\label{sec:apk}
The main goals in this section are to prove Theorem B (i.e.~determining the value of $br(ap^k)$) and to prove Theorem~\ref{thm:T-apk} (i.e.~showing that $\mathcal{E}_{ap^k}=\{1,2,\ldots, br(ap^k)\}$). As already remarked in the introduction, these two results play the role of base cases for Theorem A. 

\medskip

From now on, let $p$ be an odd prime. The case when $k=0$ is straightforward and is described in the following proposition.

\begin{prop}\label{prop:k=0}
Let $a\in\{1,2,\ldots, p-1\}$. Then $br(a)=f(2a)$ and $\mathcal{E}_a=\{1,2,\ldots,br(a)\}$.
\end{prop}
\begin{proof}
Every partition of $a-1$ is a $p'$-partition, and we can always construct a partition $\lambda$ of $a$ such that $|\lambda^-|=m$ for any $1\le m\le f(2a)$, since $f(2a)$ is the maximum number of parts of distinct size achieved by a partition of $a$.
\end{proof}

In the following proposition we provide a naive upper bound for $br(ap^k)$, for all $k\in\mathbb{N}$ and $a\in\{1,\ldots, p-1\}$. As we will show in the rest of this section, this bound turns out to be tight for almost all values of $a$ and $k$. 

\begin{prop}\label{prop:brapk}
Let $a\in\{1,2,\ldots, p-1\}$ and let $k\in\mathbb{N}$. Then $br(ap^k)\leq 2a$.
\end{prop}
\begin{proof}
Let $C$ and $D$ be $p^k$-abacus configurations such that $D$ is obtained from $C$ by performing a single down-move. It is easy to see that the number of removable beads in $D$ is at most the number of removable beads in $C$ plus two. Hence  if $\lambda$ is a partition such that $C_{p^k}(\lambda)=\emptyset$ then $|\lambda^-|\leq 2w_{p^k}(\lambda)$. 
Now let $n=ap^k$ and let $\lambda\vdash_{p'}n$ be such that $|\lambda_{p'}^-|=br(n)$. From Theorem~\ref{thm:oddcriterion} we know that $C_{p^k}(\lambda)=\emptyset$ and $w_{p^k}(\lambda)=a$. The result follows. 
\end{proof}


To complete the proof of Theorem B, it will be convenient to split the remainder of this section into two parts. 
In each part we will appropriately fix the natural numbers $a$ and $k$ according to the statement of Theorem B. 

\subsection{Part I}
In this first part, we consider the case $k=1$ and $a<\tfrac{p}{2}$, and the case $k\ge 2$.

\begin{prop}\label{prop:brapk, 2a case}
Let $a\in\{1,2,\ldots, p-1\}$ and let $k\in\mathbb{N}$. If $k=1$ and $a<\tfrac{p}{2}$ or if $k\geq 2$, then $br(ap^k)=2a$.
\end{prop} 
\begin{proof}
From Proposition~\ref{prop:brapk} we have that $br(ap^k)\leq 2a$. Hence it is enough to construct $\lambda\vdash_{p'}ap^k$ such that $|\lambda_{p'}^-|=2a$. This is done as follows.

\noindent\textbf{(i)} First suppose that $k=1$ and $a<\tfrac{p}{2}$. Let $\lambda$ be the partition of $ap$ defined by $$\lambda=(p-1, p-2, \ldots, p-a, a, a-1, \ldots, 2, 1).$$ 
The following diagram is the $p$-abacus configuration for $\lambda$ having first gap in position $(0,0)$, where we have indicated the row numbers on the left and the runner numbers above each column:
{\footnotesize
\[ \begin{array}{lccccccccccc}
&& 0 & 1 & 2 & 3 & \cdots & 2a-2 & 2a-1 & 2a & \cdots & p-1\\
-1 && \times & \times & \times & \times & \cdots & \times & \times & \times & \cdots & \times\\
0 && \circ & \times & \circ & \times & \cdots & \circ & \times & \circ & \cdots & \circ\\
1 && \times & \circ & \times & \circ & \cdots & \times & \circ & \circ & \cdots & \circ
\end{array} \]
}

Since $C_p(\lambda)=\emptyset$ we have that $\lambda\vdash_{p'} ap$, by Theorem~\ref{thm:oddcriterion}. Moreover, we observe that $\lambda^-_{p'}=\lambda^-$ by Lemma~\ref{lem:weightgeneral}, and so $|\lambda^-_{p'}|=2a$.

\noindent\textbf{(ii)} Suppose now that $k\geq 2$.  Let $r=p^{k-1}-a>0$ and  
let $\lambda^j=a+p-2+rp+(a-j)(p-1)=p^k-(j-1)(p-1)-1$, for each $j\in\{1,2,\ldots, a\}$. Let $\lambda$ be the partition of $ap^k$ defined by 
$$\lambda=\big(\lambda^1,\lambda^2,\ldots, \lambda^a, a, (a-1)^{p-1}, (a-2)^{p-1},\ldots, 2^{p-1},1^{p-1}\big).$$
The best way to verify that $\lambda$ has the required properties is to look at it on James' abacus. We describe and depict below a $p$-abacus configuration $A$ corresponding to $\lambda$: 

\smallskip

\textbf{-} The first gap is in position $(1,0)$;

\textbf{-} Rows $1\le i\le a-1$ have a gap only in position $(i,0)$;

\textbf{-} Row $a$ has a bead only in position $(a,1)$;

\textbf{-} Rows $a+1$ to $a+r$ are all empty;

\textbf{-} Rows $a+1+r\le i\le 2a+r$ have a bead only in position $(i,0)$;

\textbf{-} There is a gap in position $(x,y)$ for all $x>2a+r$.

{\footnotesize
\[ \begin{array}{lcccccc}
&& 0 & 1 & 2 & \cdots & p-1\\
1 && \circ & \times & \times & \cdots & \times\\
\vdots &&\vdots & & & &\vdots\\
a-1 &&\circ & \times & \times & \cdots & \times\\
a &&\circ & \times & \circ & \cdots & \circ\\
a+1 &&\circ & \circ & \circ & \cdots & \circ\\
\vdots &&\vdots & & & &\vdots\\
a+r &&\circ & \circ & \circ & \cdots & \circ\\
a+1+r &&\times & \circ & \circ & \cdots & \circ\\
\vdots &&\vdots & & & &\vdots\\
2a+r &&\times & \circ & \circ & \cdots & \circ\\
\end{array} \]
}

\medskip

From the structure of $A$ we observe that 
$Q_p(\lambda)=(\lambda_0,\emptyset,\ldots, \emptyset)$, where $$\lambda_0=(\underbrace{p^{k-1},\ldots, p^{k-1}}_{a\ \text{times}}).$$
From the discussion in Section~\ref{subs:comb} (or \cite[Theorem 3.3]{OlssonBook}), we deduce that $w_{p^k}(\lambda)=w_{p^{k-1}}(\lambda_0)=a$ and $C_{p^k}(\lambda)=\emptyset$.
This shows that $\lambda\vdash_{p'}ap^k$, by Theorem~\ref{thm:oddcriterion}. 

Notice that $\lambda$ has exactly $2a$ removable nodes, corresponding to the $2a$ removable beads in $A$ lying in positions $(i,1)$ and $(a+r+i,0)$ for $i\in\{1,\dotsc,a\}$.
Let $c$ be a removable bead in position $(i,1)$ of $A$, for some $i\in\{1,\ldots, a\}$. Then $A^{\leftarrow c}$ corresponds to the partition $\mu\vdash ap^k-1$ such that $C_{p}(\mu)=(p-1)\vdash p-1$ and 
$Q_p(\mu)=(\mu_0,\mu_1,\emptyset,\ldots, \emptyset)$, where 
$$\mu_0=(\underbrace{p^{k-1}-1,\ldots, p^{k-1}-1}_{a\ \text{times}}, i-1)\ \ \text{and}\ \ \mu_1=(1^{a-i}).$$
We observe that $\mu_0\vdash_{p'}(a-1)p^{k-1}+m$, where $m:=p^{k-1}-a+(i-1)$. 
This follows from Theorem~\ref{thm:oddcriterion}, since $w_{p^{k-1}}(\mu_0)=a-1$ and $C_{p^{k-1}}(\mu_0)=(m)\vdash_{p'}m$.
Moreover, we clearly have that $\mu_1\vdash_{p'}a-i$. 
We can now use Theorem~\ref{thm:oddcrit2} to deduce that $\mu\vdash_{p'}ap^k-1$ and therefore $\mu\in\lambda_{p'}^-$. 

A similar argument shows that for every $j\in\{1,\ldots, a\}$ the $p$-abacus configuration $A^{\leftarrow d}$ obtained from $A$ by sliding the bead $d$ in position $(a+r+j,0)$ to position $(a+r+j-1,p-1)$, corresponds to a $p'$-partition $\mu$ of $ap^k-1$, that is, $\mu\in\lambda_{p'}^-$. Thus $|\lambda_{p'}^-|=2a$.
\end{proof}

\begin{prop}\label{prop:consec, 2a case}
Let $a\in\{1,2,\ldots, p-1\}$ and let $k\in\mathbb{N}$. If $k=1$ and $a<\tfrac{p}{2}$ or if $k\geq 2$, then $\mathcal{E}_{ap^k}=\{1,2,\ldots br(ap^k)\}.$
\end{prop}
\begin{proof}
From Proposition~\ref{prop:brapk, 2a case}, we have that $br(ap^k)=2a=\max\mathcal{E}_{ap^k}$. Hence it is enough to construct $\lambda\vdash_{p'}ap^k$ such that $|\lambda^-_{p'}|=m$ for each $m\in\{1,2,\dotsc,2a-1\}$. This is done as follows.

\noindent\textbf{(i)} First suppose that $k=1$ and $a<\tfrac{p}{2}$. 
We first exhibit $\lambda(j)\vdash_{p'}ap$ such that $|\lambda(j)^-_{p'}|=2j$, for each $j\in\{1,2,\dotsc,a-1\}$:

\smallskip

\textbf{-} Let $\lambda(1)=(ap-1,1)$;

\textbf{-} For each fixed $j\in\{2,\dotsc,a-1\}$, let $\lambda(j)=(\lambda_1,\lambda_2,\dotsc,\lambda_{2j})$ where 

\textbf{\hspace{5pt}$\circ$} $\lambda_1=(a-j+1)p-2j+1$, 

\textbf{\hspace{5pt}$\circ$} $\lambda_x=p+2-x$ for $x\in\{2,\dotsc,j\}$, and

\textbf{\hspace{5pt}$\circ$} $\lambda_y=2j+1-y$ for $y\in\{j+1,\dotsc,2j\}$.

\smallskip

For convenience, we depict the $p$-abacus of $\lambda(j)$ having first gap in position $(0,0)$:
{\footnotesize
\[ \begin{array}{lcccccccccccc}
&& 0 & 1 & 2 & 3 & \cdots & 2j-2 & 2j-1 & 2j & \cdots & p-1\\
-1 && \times & \times & \times & \times & \cdots & \times & \times & \times & \cdots & \times\\
0 && \circ & \times & \circ & \times & \cdots & \circ & \times & \circ & \cdots & \circ\\
1 && \circ & \circ & \times & \circ & \cdots & \times & \circ & \circ & \cdots & \circ\\
2 && \circ & \circ & \circ & \circ & \cdots & \circ & \circ & \circ & \cdots & \circ\\
\vdots && \vdots &&&&&&&&& \vdots \\
a-j+1 && \times & \circ & \circ & \circ & \cdots & \circ & \circ & \circ & \cdots & \circ\\
\end{array} \]
}

By Theorem~\ref{thm:oddcriterion}, we have $\lambda(j)\vdash_{p'}ap$, and by Lemma~\ref{lem:weightgeneral}, we have $|\lambda(j)^-_{p'}|=|\lambda(j)^-|=2j$. Hence $\{2,4,\dotsc,2a-2\}\subseteq \mathcal{E}_{ap}$.

Next we exhibit $\beta(j)\vdash_{p'}ap$ such that $|\beta(j)^-_{p'}|=2j-1$ for each $j\in\{1,2,\dotsc,a\}$:

\smallskip

\textbf{-} Let $\beta(1)=((a-1)p+1,1^{p-1})$;

\textbf{-} Let $\beta(a)=(2a-1,2a-2,\dotsc,a+1,a^{p-2a+2},a-1,\dotsc,2,1)$;

\textbf{-} For each fixed $j\in\{2,\dotsc,a-1\}$, let $\beta(j)=(\beta_1,\dotsc,\beta_p)$ where

\textbf{\hspace{5pt}$\circ$} $\beta_1=(a-j)p+1$,

\textbf{\hspace{5pt}$\circ$} $\beta_x=2j+2-x$ for $x\in\{2,\dotsc,j\}$, 

\textbf{\hspace{5pt}$\circ$} $\beta_y=j$ for $y\in\{j+1,\dotsc,p-j+1\}$, and 

\textbf{\hspace{5pt}$\circ$} $\beta_z=p+1-z$ for $z\in\{p-j+2,\dotsc,p\}$.

\smallskip

For convenience, we depict the $p$-abacus of $\beta(j)$ having first gap in position $(0,0)$:
{\footnotesize
\[ \begin{array}{lcccccccccccc}
&& 0 & 1 & 2 & 3 & \cdots & 2j-2 & 2j-1 & 2j & \cdots & p-1\\
-1 && \times & \times & \times & \times & \cdots & \times & \times & \times & \cdots & \times\\
0 && \circ & \times & \circ & \times & \cdots & \circ & \times & \times & \cdots & \times\\
1 && \circ & \circ & \times & \circ & \cdots & \times & \circ & \circ & \cdots & \circ\\
2 && \circ & \circ & \circ & \circ & \cdots & \circ & \circ & \circ & \cdots & \circ\\
\vdots && \vdots &&&&&&&&& \vdots \\
a-j+1 && \times & \circ & \circ & \circ & \cdots & \circ & \circ & \circ & \cdots & \circ\\
\end{array} \]
}

Again by Theorem~\ref{thm:oddcriterion} we have $\beta(j)\vdash_{p'}ap$. By Lemma~\ref{lem:weightgeneral}, if $j\ne a$ then $|\beta(j)^-|=2j$ and $|\beta(j)^-\setminus \beta(j)^-_{p'}|=1$, 
while if $j=a$ then $|\beta(j)^-_{p'}|=|\beta(j)^-|=2a-1$. 
In both cases we have $|\beta(j)^-_{p'}|=2j-1$, giving $\{1,3,\dotsc,2a-1\}\subseteq \mathcal{E}_{ap}$. Thus $\mathcal{E}_{ap}=\{1,2,\dotsc,2a\}$ as claimed.

\medskip

\noindent\textbf{(ii)} Suppose now that $k\ge 2$. We first construct a partition $\lambda(j)\vdash_{p'}ap^k$ such that $|\lambda(j)^-_{p'}|=2a-j$, for all $j\in\{1,2,\ldots, a-1\}$. Let $r=p^{k-1}-a>0$ and let 
$$\lambda(j):=\big(\eta_{a-1},\ldots, \eta_j,\theta_j,\ldots, \theta_1,a,(a-1)^{p-1},\ldots, (j+1)^{p-1},j^{p-2},(j-1)^{p-1},\ldots, 1^{p-1}\big),$$
where $\theta_t=a+pr+t(p-1)$ and $\eta_t=\theta_t+(p-2)$, for any $t\in\{1,\ldots, a-1\}$.
As usual, it is useful to look at $\lambda(j)$ on James' abacus. We describe and depict below a $p$-abacus $A^j$ of $\lambda(j)$:

\smallskip

\textbf{-} The first gap is in position $(1,1)$;

\textbf{-} Rows $1\le x\le j$ have a gap only in position $(x,1)$;

\textbf{-} Rows $j+1\le x\le a-1$ have a gap only in position $(x,0)$;

\textbf{-} Row $a$ has a bead only in position $(a,1)$;

\textbf{-} Rows $a+1$ to $a+r$ are all empty;

\textbf{-} Rows $a+r+1\le x\le a+r+j$ have a bead only in position $(x,1)$;

\textbf{-} Rows $a+r+j+1\le x\le 2a+r$ have a bead only in position $(x,0)$;

\textbf{-} There is a gap in position $(x,y)$ for all $x>2a+r$.

{\footnotesize
\[ \begin{array}{lcccccccc}
&&& 0 & 1 & 2 & \cdots & p-1\\
1&&&\times & \circ & \times & \cdots & \times\\
\vdots &&&\vdots & & & & \vdots\\
j&&&\times & \circ & \times & \cdots & \times\\
j+1&&&\circ & \times & \times & \cdots & \times\\
\vdots &&&\vdots & & & & \vdots\\
a-1&&&\circ & \times & \times & \cdots & \times\\
a&&&\circ & \times & \circ & \cdots & \circ\\
a+1&&&\circ & \circ & \circ & \cdots & \circ\\
\vdots &&&\vdots & & & & \vdots\\
a+r&&&\circ & \circ & \circ & \cdots & \circ\\
a+r+1&&&\circ & \times & \circ & \cdots & \circ\\
\vdots &&&\vdots & & & & \vdots\\
a+r+j&&&\circ & \times & \circ & \cdots & \circ\\
a+r+j+1&&&\times & \circ & \circ & \cdots & \circ\\
\vdots &&&\vdots & & & & \vdots\\
2a+r&&&\times & \circ & \circ & \cdots & \circ\\
\end{array} \]
}

Since $j$ is fixed, we will denote $\lambda(j)$ by $\lambda$ and $A^j$ by $A$ from now on.
Arguing as in the proof of Proposition~\ref{prop:brapk, 2a case}, we deduce that $\lambda\vdash_{p'}ap^k$. Moreover, it is clear that $|\lambda^-|=2a$.
Let $x\in\{1,\ldots j\}$ and let $c$ be the bead lying in position $(x,2)$ of $A$. Let $\mu^x$ be the partition of $ap^k-1$ corresponding to the $p$-abacus $A^{\leftarrow c}$. Then $C_{p}(\mu^x)=(p,1^{p-1})$. Therefore $\mu^x$ is not a $p'$-partition, by Theorem~\ref{thm:oddcrit2}. It follows that $|\lambda_{p'}^-|\leq 2a-j$.

We will now show that all of the other $2a-j$ removable beads in $A$ correspond to $p'$-partitions of $ap^k-1$. Let $x\in\{j+1,j+2,\ldots, a\}$ and let $c$ be the bead in position $(x,1)$ of $A$. Let $\mu^x$ be the partition of $ap^k-1$ corresponding to the $p$-abacus $A^{\leftarrow c}$. Then $C_{p}(\mu^x)=(p-1)\vdash_{p'}p-1$ and $Q_{p}(\mu^x)=(\mu_0,\mu_1,\emptyset,\ldots, \emptyset)$, where
$$\mu_0=\big((p^{k-1}-1)^{a-j}, x-j-1\big)\ \ \text{and}\ \ \mu_1=\big((r+j+1)^j, (j+1)^{a-x}, j^{x-j-1}\big).$$
By Theorem~\ref{thm:oddcriterion}, we have that $\mu_0\vdash_{p'}|\mu_0|$ and $\mu_1\vdash_{p'}|\mu_1|$, where 
$$|\mu_0|=(a-j-1)p^{k-1}+(p-1)\sum_{i=1}^{k-2}p^i+[(p-1)-(a-x)]$$ and $$|\mu_1|=jp^{k-1}+(a-x).$$
This implies $\mu^x\vdash_{p'}ap^k-1$, by Theorem~\ref{thm:oddcrit2}.

Now let $c$ be the bead in position $(a+r+x,1)$ for some $x\in\{1,\ldots,j\}$, and let $\mu^x$ be the partition corresponding to the $p$-abacus $A^{\leftarrow c}$. Arguing as before, we deduce from Theorem~\ref{thm:oddcrit2} that $\mu^x\vdash_{p'}ap^k-1$.

Finally, let $c$ be the bead in position $(a+r+x,0)$ for some $x\in\{j+1,\ldots, a\}$ and let $\mu^x$ be the partition corresponding to $A^{\leftarrow c}$. First, we observe that $C_{p}(\mu^x)=(p-2,1)\vdash_{p'}p-1$.
Moreover, $Q_{p}(\mu^x)=(\mu_0,\mu_1,\emptyset,\ldots, \emptyset, \mu_{p-1})$, where
$$\mu_0=\big((p^{k-1}+1)^{a-x}, (p^{k-1})^{x-j-1}\big),\ \ \ \mu_1=\big((r+j)^j, j^{a-j}\big),\quad \mathrm{and}\quad \mu_{p-1}=(r+x-1).$$
Again, we use Theorem~\ref{thm:oddcrit2} to deduce that $\mu^x\vdash_{p'}ap^k-1$. Thus $|\lambda_{p'}^-|=2a-j$, as desired, and so $\{a+1,a+2,\ldots, 2a-1\}\subseteq \mathcal{E}_{ap^k}$. 

\medskip

Finally, we construct a partition $\beta(j)\vdash_{p'}ap^k$ such that $|\beta(j)^-_{p'}|=a-j$, for all $j\in\{0,1,\ldots, a-1\}$. Let $B^j$ be the $p$-abacus configuration obtained from the $p$-abacus $A^j$ described above by removing the bead in position $(a,1)$ so that row $a$ is now empty.
Let $\beta(j)$ be the partition of $ap^k$ corresponding to the $p$-abacus $B^j$. Again, since $j$ is fixed we will now denote $B^j$ by $B$ and $\beta(j)$ by $\beta$.

It is clear that $\beta\vdash_{p'}ap^k$ and that $|\beta^-|=2a-j-1$. Moreover, if $c$ is one of the $a-1$ removable beads lying on runner $1$ of $B$ and $\mu$ is the partition of $ap^k-1$ corresponding to the $p$-abacus $B^{\leftarrow c}$, then $C_{p}(\mu)=(p,1^{p-1})$ and therefore $\mu$ is not a $p'$-partition, by Theorem~\ref{thm:oddcrit2}. Hence $|\beta^-_{p'}|\leq a-j$.
Arguing as before, the partition corresponding to the $p$-abacus $B^{\leftarrow c}$ for any removable bead $c$ lying on runner $0$ of $B$ is a $p'$-partition of $ap^k-1$. Hence $|\beta^-_{p'}|= a-j$, and so $\{1,2,\dotsc,a\}\subseteq \mathcal{E}_{ap^k}$. Thus $\mathcal{E}_{ap^k}=\{1,2,\dotsc,2a\}$ as claimed.
\end{proof}

\subsection{Part II}

In this second part of Section~\ref{sec:apk}, we fix $k=1$ and $a\in\mathbb{N}$ such that $\tfrac{p}{2}<a<p$. The main aim in Part II is to prove the following fact.

\begin{prop}\label{prop:brapk, floor case}
Let $a\in\mathbb{N}$ be such that $\tfrac{p}{2}<a<p$. Then $br(ap)=p-1+2\lfloor\tfrac{2a-(p-1)}{6}\rfloor$.
\end{prop}

The proof of Proposition~\ref{prop:brapk, floor case} is split into a series of technical lemmas. We start by fixing some notation which will be kept throughout Part II. 

\begin{notation}\label{not:x/3}
Let $a\in\mathbb{N}$ be such that $\tfrac{p}{2}<a<p$. We let $x:=a-\frac{p-1}{2}$, and we write 
$x=3q+\delta$ for some 
$q\in\mathbb{N}_0$ and $\delta\in\{0,1,2\}$. In particular we have $q=\lfloor\frac{x}{3} \rfloor=\lfloor\tfrac{2a-(p-1)}{6}\rfloor$.
\end{notation}

\begin{defn}
Denote by $A_\emptyset$ the $p$-abacus configuration for the empty partition $\emptyset$, such that $A_\emptyset$ has first gap in position $(0,0)$. 
We then define $\mathcal{Z}(a)$ to be the set of $p$-abaci $B$ such that $w(B)=a$ and $B^\uparrow=A_\emptyset$. 
It is clear by Theorem~\ref{thm:oddcriterion} that $\mathcal{Z}(a)$ is naturally in bijection with $\mathrm{Irr}_{p'}(\fS_{ap})$. 
\end{defn}

\begin{lem}\label{lem:Za}
Let $\lambda\vdash_{p'}ap$ and let $B\in\mathcal{Z}(a)$ be the $p$-abacus corresponding to $\lambda$. Then $$|\lambda^-_{p'}|=\sum_{i=1}^{p-1}\mathrm{Rem}(B_i)\quad \ \text{and}\quad \ \ br(ap)=\max_{B\in\mathcal{Z}(a)}\ \sum_{i=1}^{p-1} \operatorname{Rem}(B_i).$$
\end{lem}
\begin{proof}
The statement follows directly from Lemma~\ref{lem:weightgeneral} and Theorem~\ref{thm:oddcriterion}.
\end{proof}


\begin{lem}\label{lem:beta}
For $a\in\mathbb{N}$ such that $\tfrac{p}{2}<a<p$, we have $br(ap)\ge p-1+2q$.
\end{lem}

\begin{proof}
We exhibit a partition $\beta\vdash_{p'}ap$ such that $|\beta^-_{p'}|=p-1+2q$. If $\delta=0$ then let $\beta$ be the following partition of $ap$:
$$(p+2q,p+2q-1,\dotsc,p+q+1,p+q-1,\dotsc,q+1,q^{p-2q+1},q-1,\dotsc,2,1),$$ while if $\delta\ne 0$ then let $\beta$ be the following partition of $ap$: $$\big(p(\delta+1)+2,p+2q+1,p+2q,\dotsc,p+q+3,p+q-1,\dotsc,q+1,q^{p-2q+1},q-1,\dotsc,1\big).$$

We describe and depict below a $p$-abacus $B_\beta\in\mathcal{Z}(a)$ of $\beta$:

\smallskip

\noindent\textbf{-} For $j\in\{0,2,\dotsc,p-3,p-1\}$, runner $j$ has beads in positions $(x,j)$, for all $x\le -1$;

\noindent\textbf{-} Runner $1$ has beads in positions $(0,1)$, $(1+\delta,1)$ and $(y,1)$ for all $y\le -3$;

\noindent\textbf{-}  For $j\in\{3,5,\dotsc,2q-1\}$, runner $j$ has beads in positions $(0,j)$, $(1,j)$ and $(y,j)$ for all $y\le -3$;

\noindent\textbf{-} For $j\in\{2q+1,2q+3,\dotsc,p-2\}$ runner $j$ has beads in positions $(0,j)$ and $(y,j)$ for all $\le -2$.

{\footnotesize
\[B_\beta : \begin{array}{lcccccccccccccccc}
&& 0 & 1 & 2 & 3 & 4 & \cdots & 2q-2 & 2q-1 & 2q & 2q+1 & 2q+2 & \cdots & p-3 & p-2 & p-1\\
-3 && \times & \times & \times & \times & \times & \cdots & \times & \times & \times & \times & \times & \cdots & \times & \times & \times\\
-2 && \times & \circ & \times & \circ & \times & \cdots & \times & \circ & \times & \times & \times & \cdots & \times & \times & \times\\
-1 && \times & \circ & \times & \circ & \times & \cdots & \times & \circ & \times & \circ & \times & \cdots & \times & \circ & \times\\
0 && \circ & \times & \circ & \times & \circ & \cdots & \circ & \times & \circ & \times & \circ & \cdots & \circ & \times & \circ\\
1 && \circ & \circ & \circ & \times & \circ & \cdots & \circ & \times & \circ & \circ & \circ & \cdots & \circ & \circ & \circ\\
2 && \circ & \circ & \circ & \circ & \circ & \cdots & \circ & \circ & \circ & \circ & \circ & \cdots & \circ & \circ & \circ\\
\vdots && \vdots &&&&&&&&&&&&&& \vdots\\
1+\delta && \circ & \times & \circ & \circ & \circ & \cdots & \circ & \circ & \circ & \circ & \circ & \cdots & \circ & \circ & \circ\\
\end{array} \]
}

We observe from the above abacus configuration that $C_p(\beta)=\emptyset$, and hence by Theorem~\ref{thm:oddcriterion} we have that $\beta\vdash_{p'}ap$. Moreover by Lemma~\ref{lem:Za}, we have $\beta^-=\beta_{p'}^-$. Hence $br(ap)\ge |\beta_{p'}^-|=p-1+2q$.
\end{proof}

Thus it remains to show that for all $\lambda\vdash_{p'}ap$, we have $|\lambda_{p'}^-|\le p-1 + 2q$. 
In order to do this we introduce a new combinatorial object. 

\begin{defn}\label{def:doubled}
Let $T_\emptyset$ be the $2$-abacus configuration for the empty partition $\emptyset$ having first gap in position $(0,0)$. Let $U^{(0)}$, $U^{(1)}\ldots, U^{(p-1)}$ be $2$-abaci such that $(U^{(i)})^{\uparrow}=T_\emptyset$ for all $i\in\{0,1,\ldots, p-1\}$. If $w(U^{(0)})+w(U^{(1)})+\cdots +w(U^{(p-1)})=w\in\mathbb{N}_0$ then we call the sequence $\underline{U}=(U^{(0)},U^{(1)},\ldots, U^{(p-1)})$ a \textit{doubled $p$-abacus of weight $w$} (we write $w(\underline{U})=w$ in this case). 
Moreover we denote by $\mathcal{D}(w)$ the set of doubled $p$-abaci of weight $w$.

Finally, given any $w\in\mathbb{N}_0$ we let
$M(w)=\mathrm{max}\big\{\rho(\underline{U})\ |\ \underline{U}\in\mathcal{D}(w)\big\},$
where for any $\underline{U}\in\mathcal{D}(w)$ we define $\rho(\underline{U})$ as $$\rho(\underline{U})=\sum_{i=1}^{p-1}\mathrm{Rem}(U_1^{(i)}).$$
As usual, we denoted by $U^{(i)}_0$ (and $U_1^{(i)}$) the left (and right) runner of the $2$-abacus $U^{(i)}$.
\end{defn}
\begin{rem}\label{rem:doubled1}
Let $\lambda\vdash_{p'} ap$ and let $B\in\mathcal{Z}(a)$ correspond to $\lambda$. For $i\in\{1,\ldots, p-1\}$, let 
$U^{(i)}=(B_{i-1}, B_{i})$, and let $U^{(0)}=(B_{p-1},B_0)$. Then $\underline{U}:=(U^{(0)},U^{(1)},\ldots, U^{(p-1)})\in\mathcal{D}(2a)$ and 
$\rho(\underline{U})=|\lambda_{p'}^-|,$
by Lemma~\ref{lem:Za}.
With this in mind we define $\mathcal{D}(\mathcal{Z}(a))$ to be the subset of $\mathcal{D}(2a)$ of sequences 
$\underline{U}:=(U^{(0)},U^{(1)},\ldots, U^{(p-1)})$ such that $U^{(i)}_0=U^{(i-1)}_1$ for all $i\in\{0,1,\ldots, p-1\}$ (here two runners are equal if they coincide as $1$-abaci; that is, they have beads in exactly the same rows). 
Clearly the set $\mathcal{D}(\mathcal{Z}(a))$ is naturally in bijection with $\mathcal{Z}(a)$, via the construction described above. 
\end{rem}

\begin{lem}\label{lem:M}
Let $a$ and $x$ be as in Notation~\ref{not:x/3}. Then $br(ap)\le M(2a) = p-1 + \lfloor\tfrac{2x}{3}\rfloor$.
\end{lem}

\begin{proof}

From Remark~\ref{rem:doubled1} it is immediate that $br(ap)\le M(2a)$, so it remains to prove $M(2a)=p-1 + \lfloor\tfrac{2x}{3}\rfloor$.

Let $\underline{U}=(U^{(0)},U^{(1)},\ldots, U^{(p-1)})\in\mathcal{D}(2a)$ be such that $\rho(\underline{U})=M(2a)$. Let $w_i=w(U^{(i)})$. Clearly $w_1+w_2+\cdots w_{p-1}\leq 2a$. Moreover, arguing as in the proof of Lemma~\ref{lem:T} we can assume that $w(U^{(i)}_1)=w_i$ and $w(U^{(i)}_0)=0$ for all $i\in\{1,\ldots, p-1\}$. From the maximality of $\rho(\underline{U})$ we deduce using Lemma~\ref{lem:T} (in the case $\lambda=\emptyset$) that 
$$\operatorname{Rem}(U^{(i)}_1) = \lfloor\sqrt{w_i}\rfloor$$ and hence
$$M(2a) = \max\left\{ \sum_{i=1}^{p-1}\lfloor \sqrt{b_i}\rfloor\ \middle|\ b_1+\cdots+b_{p-1}\le 2a\ \mathrm{and}\ b_i\in\mathbb{N}_0\ \ \forall\ 1\le i\le p-1 \right\}.$$


Let $b=(b_1,\dotsc,b_{p-1})$ be such that $b_i\in\mathbb{N}_0$ for all $i$, $\sum_i b_i\le 2a$ and $\sum_i\lfloor \sqrt{b_i}\rfloor = M(2a)$; we will call any $(p-1)$-tuple satisfying these conditions \emph{maximal}. If there exists $i$ such that $b_i\ge 9$, then there exists $j$ such that $b_j\le 1$. This follows since $\sum b_i\le 2a<2p$. 
Replacing $b_i$ by $b'_i=b_i-4$ and $b_j$ by $b_j'=b_j+4$ in $b$ we obtain a new maximal sequence $b'$. 
Hence we may assume without loss of generality that our maximal sequence $b$ has $b_i\le 8$ for all $i\in\{1,\ldots, p-1\}$. 

Now if there exists $i$ such that $b_i=0$ then there exists $j$ such that $b_j\ge 2$, because $2a>p$. In this case, replacing $b_i$ by $b_i'=1$ and $b_j'=b_j-1$ in $b$ we obtain a new maximal sequence $b'$. 
Hence we may further assume that $b$ has $b_i\ge 1$ for all $i\in\{1,\ldots, p-1\}$. The observations above show that without loss of generality we may assume
$$\lfloor\sqrt{b_1}\rfloor=\cdots=\lfloor\sqrt{b_t}\rfloor=2,\ \lfloor\sqrt{b_{t+1}}\rfloor=\cdots=\lfloor\sqrt{b_{p-1}}\rfloor=1,$$
for some $t\in\{0,\dotsc,p-1\}$. 

In particular, $b_i\in\{4,\dotsc,8\}$ for $i\in\{1,\dotsc,t\}$ and $b_j\in\{1,2,3\}$ for $j\in\{t+1,\dotsc,p-1\}$. Thus $4t+(p-1-t)\le \sum b_i\le 2a$, which gives $t\le \lfloor\tfrac{2x}{3}\rfloor$ since $t$ is an integer. This in turn implies that $M(2a)=2t+(p-1-t)\le p-1+\lfloor\tfrac{2x}{3}\rfloor$. 

Finally, equality holds because we can construct $\underline{U}\in\mathcal{D}(2a)$ such that $w(U^{(1)})=\cdots =w(U^{(t)})=4$, $w(U^{(t+1)})=\cdots =w(U^{(p-1)})=1$ and $w(U^{(0)})=2a-3t-(p-1)$,  where  $t=\lfloor\tfrac{2x}{3}\rfloor$, with $\mathrm{Rem}(U_1^{(j)})=2$ for $j\in\{1,\dotsc,t\}$ and $\mathrm{Rem}(U_1^{(j)})=1$ for $j\in\{t+1,\dotsc,p-1\}$. 
\end{proof}

Lemmas~\ref{lem:beta} and~\ref{lem:M} show that $p-1+2\lfloor\frac{x}{3}\rfloor\leq br(ap)\leq p-1+\lfloor\frac{2x}{3}\rfloor.$
In particular if $\delta\neq 2$ then we have that $\lfloor\tfrac{2x}{3}\rfloor=2q+ \lfloor\frac{2\delta}{3}\rfloor=2q=2\lfloor\tfrac{x}{3}\rfloor.$
In this case we have $br(ap)=M(2a)=p-1+2q$. To deal with the remaining case of $\delta=2$ where $p-1+2q\le br(ap)\le M(2a)=p-1+2q+1$, we have the following lemma. 

\begin{lem}\label{lem:delta=2}
Let $a\in\mathbb{N}$ be as in Notation~\ref{not:x/3} and suppose that $\delta=2$. 
Then $br(ap) \le M(2a)-1$.
\end{lem}

\begin{proof}
From Remark~\ref{rem:doubled1} it is enough to show that if
$\underline{U}\in\mathcal{D}(2a)$ and $\rho(\underline{U})=M(2a)$, then $\underline{U}\notin\mathcal{D}(\mathcal{Z}(a)).$ To do this we will show that if $\rho(\underline{U})=M(2a)$ then there exists $i\in\{0,1,\ldots, p-1\}$, such that $U^{(i)}_0\neq U^{(i-1)}_1$.

For $i\in\{0,1,\ldots, p-1\}$, let $b_i=w(U^{(i)})$. Arguing as in the proof of Lemma~\ref{lem:M} we see that $\rho(\underline{U})=\sum_{i=1}^{p-1}\lfloor\sqrt{b_i}\rfloor$.
Moreover, given any composition $\underline{w}=(w_1,\ldots, w_{p-1})$ such that $w_1+\cdots + w_{p-1}\leq 2a$ there exists $\underline{V}\in\mathcal{D}(2a)$ such that $w(V^{i})=w_i$ for all $i\in\{1,\ldots, p-1\}$, $w(V^{0})=2a-(w_1+\cdots+w_{p-1})$ and such that $\rho(\underline{V})=\sum_{i=1}^{p-1}\lfloor\sqrt{w_i}\rfloor$. 

Let $\underline{b}=(b_1,\ldots, b_{p-1})$ and suppose that $b_i\ge 9$, for some $i\in\{1,\ldots p-1\}.$
\begin{itemize}
\item[-] If there exists $j$ such that $b_j=0$, then replacing $(b_i,b_j)$ by $(b'_i,b'_j):=(b_i-4,4)$
in $\underline{b}$ we obtain a new composition $\underline{b'}$ such that $\sum_{i=1}^{p-1}\lfloor\sqrt{w_i}\rfloor >\rho(\underline{U})$. This clearly contradicts the maximality of $\rho(\underline{U})$. 
\item[-] If $b_i \ge 10$, then there exists $j\ne l$ such that $b_j=b_l=1$ since $a<p$. 
But then we may replace $(b_i,b_j,b_l)$ by $(b_i-6,4,4)$ in $\underline{b}$ to obtain a similar contradiction as before. 
\item[-] If there exists $i'\ne i$ such that $b_{i'}\ge 9$, then as above we deduce that $b_{i'}=9$. In particular, $2a\ge 18$ so $p>3$. Since $a<p$, there exist distinct $j,j',j''$ such that $b_j=b_{j'}=b_{j''}=1$. 
But then we may replace $(9,9,1,1,1)$ by $(5,4,4,4,4)$ in $\underline{b}$ to obtain a contradiction.
\end{itemize}
The above observations show that if $b_i=9$ for some $i\in\{1,\ldots, p-1\}$ then there exists $t\in\{0,1,\ldots, p-2\}$ such that $\underline{b}$ has $t$ parts satisfying $\lfloor\sqrt{b_j}\rfloor=2$ and $p-2-t$ parts satisfying $\lfloor\sqrt{b_j}\rfloor=1$. Hence $M(2a)=3+2t+(p-2-t)=p-1+\lfloor\tfrac{2x}{3}\rfloor=p-1+2q+1,$ so $t=2q-1$. But this implies that $$2a \ge \sum_{m=1}^{p-1} b_m \ge 9 + 4t + (p-2-t) = p-1+6q+5.$$ Therefore we have $6q+5\le 2a-(p-1) = 2x = 6q+4$, a contradiction. Thus, $\lfloor\sqrt{b_i}\rfloor\in\{0,1,2\}$ for each $i\in\{1,\dotsc,p-1\}$.

So suppose there are $t$ values of $i$ for which $\lfloor\sqrt{b_i}\rfloor=2$, $s$ values for which it is 1, and $p-1-s-t$ values for which it is 0. Then $$p+2q = M(2a) = 2t+s \le p-1+t,$$
so $t\ge 2q+1$. In particular $t\ge 1$, so there exists $i$ with $\lfloor\sqrt{b_i}\rfloor=2$. If there exists $j\ne l$ such that $b_j=b_l=0$, then we may replace $(b_i,b_j,b_l)$ by $(b_i-2,1,1)$ in $\underline{b}$ to obtain a contradiction to the maximality of $\rho(\underline{U})$. So there is at most one $b_j=0$ and thus $s+t\in\{p-2,p-1\}$. 

If $s+t=p-2$, then $p+2q=M(2a)=2t+s$ implies $t=2q+2$, and so 
$$6q+4-b_0 = 2x-b_0= \sum_{m=1}^{p-1}b_m-(p-1)\ge 4t+s-(p-1)=6q+5,$$
which is a contradiction. Thus $s+t=p-1$ and $t=2q+1$. Since
$$6q+4-b_0=\sum_{m=1}^{p-1}b_m-(p-1) \ge 4t+s-(p-1)=6q+3,$$
one of the following must hold:
\begin{itemize}
\item[(i)] $|\{i:b_i=4\}|=t$, $|\{i:b_i=1\}|=s$ and $b_0=1$; or
\item[(ii)] $|\{i:b_i=4\}|=t-1$, $|\{i:b_i=5\}|=1$, $|\{i:b_i=1\}|=s$ and $b_0=0$; or
\item[(iii)] $|\{i:b_i=4\}|=t$, $|\{i:b_i=2\}|=1$, $|\{i:b_i=1\}|=s-1$ and $b_0=0$.
\end{itemize}

Now, suppose for a contradiction that $\underline{U}\in\mathcal{D}(\mathcal{Z}(a))$. Then we have that the bead configurations on $U^{(i-1)}_1$ and $U^{(i)}_0$ are equal for all $j$; call this property $(\star)$. The key in the following will be that $t= |\{i:b_i\ge 4\}| = 2q+1$ is odd.

In case (i), let $i\in\{1,\ldots, p-1\}$ be such that $b_i=4$. Then $(w(U^{(i)}_0), w(U^{(i)}_1))=(j,4-j)$ for some $0\le j\le 4$. If $j=2$ then $(\star)$ would imply $b_{i+1}\ge 2$, and hence $b_{i+1}=4$. Moreover it also forces $w(U^{(i+1)}_0)=w(U^{(i+1)}_1)=2$. We can iterate this argument to deduce that $w(U^{(y)}_0)=w(U^{(y)}_1)=2$ for all $y\in\{0,1,\ldots, p-1\}$, which gives a contradiction. Thus $j\in\{0,1,3,4\}$.

If $j=0$, then $w(U^{(i)}_1)=4$, so $(\star)$ implies that $w(U^{(i+1)}_0)=4$ and hence $b_{i+1}=4$ also. 
Similarly if $j=1$, then $w(U^{(i+1)}_0)=3$ and hence $b_{i+1}=4$. 
On the other hand, if $j=3$ or $j=4$ then similarly we deduce that $b_{i-1}=4$. These observations imply that $t$ is an even natural number (because if $j\in\{0,1\}$ then we may pair off $i$ and $i+1$ where $b_i=b_{i+1}=4$, and if $j\in\{3,4\}$ then we may pair off $i$ and $i-1$ where $b_i=b_{i-1}=4$). This gives a contradiction, and so $\underline{U}\notin\mathcal{D}(\mathcal{Z}(a))$, as desired. The analyses of cases (ii) and (iii) are similar.
%
\end{proof}

Thus when $\delta=2$ we also have that $br(ap)=p-1+2\lfloor\tfrac{x}{3}\rfloor$, by Lemmas~\ref{lem:beta}, Lemma~\ref{lem:M} and~\ref{lem:delta=2}. This proves Proposition~\ref{prop:brapk, floor case}.

\begin{proof}[Proof of Theorem B]
This follows directly from Propositions~\ref{prop:k=0},~\ref{prop:brapk, 2a case} and~\ref{prop:brapk, floor case}.
\end{proof}

We devote the final part of this section to the description of the set $\mathcal{E}_{ap}$ for any $\tfrac{p}{2}<a<p$.

\begin{prop}\label{prop:consec, floor case}
Let $a\in\mathbb{N}$ be such that $\tfrac{p}{2}<a<p$. Then $\mathcal{E}_{ap}=\{1,2,\dotsc,br(ap)\}$.
\end{prop}
\begin{proof}
Let $\beta\vdash_{p'}ap$ with $p$-abacus configuration $B:=B_\beta$ be as defined in Lemma~\ref{lem:beta}. 
In particular we proved that $|\beta^-_{p'}|=br(ap)=p-1+2q$, with notation as in Notation~\ref{not:x/3}.

Denote by $b$ the bead in position $(1+\delta,1)$ of $B$. 
For $i\in\{1,2,\ldots, \frac{p-1}{2}\}$ let $c_i$ be the bead in position $(0, p-2i)$ of $B$ and let $B(i)$ be the $p$-abacus configuration obtained from $B$ by sliding $b$ down to position $(1+\delta+i, 1)$ and by sliding $c_j$ up to position $(-1, p-2j)$ for all $j\in\{1,\ldots, i\}$. Let $\mu(i)\vdash ap$ be the partition corresponding to the $p$-abacus configuration $B(i)$.
From Theorem~\ref{thm:oddcriterion} we have that $\mu(i)\vdash_{p'}ap$ and $|\mu(i)_{p'}^-|=|\beta_{p'}^-|-2i$. It follows that $$ \{2q, 2q+2, \cdots,br(ap)-2,br(ap)\}\subseteq \mathcal{E}_{ap}.$$

Now let $A:=B(\frac{p-1}{2})$. For $i\in\{1,2,\ldots, q-1\}$ let $A(i)$ be the $p$-abacus configuration obtained from $A$ by sliding down bead $b$ from position $(1+\delta+\frac{p-1}{2}, 1)$ to position  
$(1+\delta+\frac{p-1}{2}+3i, 1)$ and by replacing runner $A_{2j+1}$ with $A_{2j+1}^{\uparrow}$ for all $j\in\{1,\ldots, i\}$. (For convenience, this step is depicted below.)

{\footnotesize
\[  \begin{array}{lcccc}
&& 2j\ \ & 2j+1 & 2j+2\\
-2 && \times & \circ & \times\\
-1 && \times & \times & \times\\
0 && \circ & \circ & \circ\\
1 && \circ & \times & \circ\\
\end{array}
\qquad\qquad\longrightarrow\qquad\qquad
\begin{array}{lcccc}
&& 2j\ \ & 2j+1 & 2j+2\\
-2 && \times & \times & \times\\
-1 && \times & \times & \times\\
0 && \circ & \circ & \circ\\
1 && \circ & \circ & \circ\\
\end{array}
\]
}

Let $\nu(i)\vdash ap$ be the partition corresponding to the $p$-abacus configuration $A(i)$. 
Since $w(A_{2i+1})=3$ for all $i\in\{1,2,\ldots, q-1\}$, it follows from Theorem~\ref{thm:oddcriterion} that $\nu(i)\vdash_{p'}ap$ and 
$|\nu(i)^-_{p'}|=|\mu(\frac{p-1}{2})^-_{p'}|-2i$. Thus $\{2, 4, 6, \cdots, 2q-2 \}\subseteq \mathcal{E}_{ap}$, and so it remains to show $\{1,3,\dotsc,br(ap)-1\}\subseteq \mathcal{E}_{ap}$. 

\medskip

First suppose $q\ge 1$. Consider the $p$-abacus configuration $C$ obtained from $B$ by sliding down the bead in position $(-1,0)$ to position $(0,0)$ and by sliding up the bead in position $(0,1)$ to position $(-1,1)$. 


Let $\gamma$ be the partition corresponding to $C$.  It is easy to see that $\gamma\vdash_{p'}ap$ and that $|\gamma^-_{p'}|=br(ap)-1$. We can now repeat the strategy used above to see that $\{3,5,\dotsc,br(ap)-1\}\subseteq \mathcal{E}_{ap}$. Of course, $1\in \mathcal{E}_{ap}$ by considering the trivial partition $(ap)\vdash_{p'}ap$.

If $q=0$ we begin with the $p$-abacus configuration $C'$ obtained from $B$ by swapping runners 0 and 1, instead of $C$.
The same argument then shows $\{1,3,\dotsc,br(ap)-1\}\subseteq \mathcal{E}_{ap}$.
\end{proof}

We conclude by observing that Propositions~\ref{prop:k=0},~\ref{prop:consec, 2a case} and~\ref{prop:consec, floor case} together prove Theorem~\ref{thm:T-apk}.

\medskip

\subsection*{Acknowledgments}
We thank Jason Long for suggesting Definition \ref{def:doubled} and for many useful conversations on this topic.



\end{document}